\documentclass[10pt]{amsart}
\usepackage{amsfonts,amssymb,amsmath, forest, a4wide}
\usepackage{amssymb, amsmath, a4wide}
\usepackage{float,epsfig}
\usepackage{xcolor}
\usepackage{ mathrsfs }
\usepackage{graphicx}
\usepackage{hyperref, dsfont, bbm}
\usepackage{mathabx}
\usepackage{multicol}
\usepackage[left=2.5cm,top=1 cm,right=2.5cm,nofoot]{geometry}

\def \G{{\mathscr G}}

\def \H{\mathcal H}
\def \l2x{L^2(X;\mc H)}

\def \dalpha{{d{\mu_{\widehat \Gamma}}(\alpha)}}
\def \dGammax{{d{\mu_{\Gamma\backslash \mathscr G}}(\Gamma x)}}

\newcommand{\mc}{\mathcal}
\newcommand{\msc}{\mathscr}
\newcommand{\mf}{\mathfrak}

\newcommand{\U}{\mathbin{\mathcal{U}\kern-.1em}}
\renewcommand{\S}{\mathbin{\mathcal{S}\kern-.08em}}

\newcommand{\ra}{\rightarrow}

\renewcommand{\ss}{\subset}
\newcommand{\ol}{\overline}

\DeclareMathOperator*{\Span}{span}
\DeclareMathOperator*{\supp}{supp}
\DeclareMathOperator*{\a.e.}{a.e.}

\newtheorem{thm}{Theorem}[section]

\newtheorem{lem}[thm]{Lemma}

\newtheorem{deff}[thm]{Definition}

\newtheorem{prop}[thm]{Proposition}
\newtheorem{cor}[thm]{Corollary}

\newtheorem{theorem}{Theorem}[section]

\newtheorem{example}[theorem]{Example}

\newenvironment{defn}{\begin{deff}
		\rm }{\end{deff}}
\newenvironment{exa}{\begin{example}
		\rm }{\end{example}}

\theoremstyle{definition}

\theoremstyle{remark}

\numberwithin{equation}{section}

\usepackage{scalerel,stackengine}
\stackMath
\newcommand\reallywidehat[1]{%
	\savestack{\tmpbox}{\stretchto{%
			\scaleto{%
				\scalerel*[\widthof{\ensuremath{#1}}]{\kern.1pt\mathchar"0362\kern.1pt}%
				{\rule{0ex}{\textheight}}
			}{\textheight}%
		}{2.4ex}}%
	\stackon[-6.9pt]{#1}{\tmpbox}%
}
\begin{document}

\title[Subspace Dual and orthogonal frames]{ Subspace Dual and orthogonal frames\\ by action of an abelian group }
\author{Sudipta Sarkar}

\author{Niraj K. Shukla}
\address{Department of Mathematics, 
	Indian Institute of Technology Indore, 
	Simrol, Khandwa Road, 
	Indore-453 552, India}
 \email{sudipta.math7@gmail.com,  o.nirajshukla@gmail.com}
\thanks{The first author was supported   by CSIR, New Delhi (09/1022(0037)/2017-EMR-I) while the second author was supported by NBHM-DAE(Grant 02011/2018-NBHM(R.P.)/R\&D II/14723)}




\begin{abstract}
In this article, we discuss subspace duals of a frame of translates by an action of a closed abelian subgroup $\Gamma$ of a locally compact group $\mathscr G.$ These subspace duals are not required to lie in the space generated by the frame. We characterise translation-generated subspace duals of a frame/Riesz basis involving the Zak transform for the pair $(\mathscr G, \Gamma) .$ We continue our discussion on the orthogonality of two translation-generated Bessel pairs using the Zak transform, which allows us to explore the dual of super-frames. As an example, we extend our findings to splines, Gabor systems, $p$-adic fields $\mathbb Q p,$ locally compact abelian groups using the fiberization map.
\end{abstract}
\maketitle
\section{Introduction}

Let $\mathscr G$ be a second countable locally compact group (may not be abelian)  and  $\Gamma$ be a closed abelian subgroup of $\mathscr G$. A closed subspace  $V$ in $L^2(\mathscr G)$ is said to be   \textit{$\Gamma$-translation invariant ($\Gamma$-TI)} if $L_\xi f \in V$  for all $f\in V$ and $\xi \in \Gamma$, where for $\eta \in \mathscr G$ the \textit{left translation} $L_\eta$ on $L^2 (\mathscr G)$ is defined by    
$$(L_\eta f)(\gamma) = f(\eta^{-1} \gamma), \quad \gamma \in \mathscr G \  \mbox{and }\ f\in L^2(\mathscr G).$$
 For a  family of functions  $\mc A :=\{\varphi_t : t\in  \mc N\}$ in $L^2(\mathscr G),$ we   define   \textit{$\Gamma$-translation generated (TG)} system $\mc E^{\Gamma} (\mc A)$ and its associated  $\Gamma$-translation invariant    space $\mc S^{\Gamma} (\mc A)$   as follows: 
\begin{align}\label{TIsystem}
\mc E^{\Gamma} (\mc A):=\{L_{\gamma}\varphi : \gamma \in  \Gamma, \varphi \in \mc A\} , \ \mbox{and} \ \mc S^{\Gamma}(\mc A) := \ol{\Span}\{L_{\gamma}\varphi : \gamma \in \Gamma,\varphi \in \mc A\},
\end{align}
respectively.  
 For $\mc A=\{\varphi\},$ we denote $\mc E^{\Gamma} (\mc A)$ and $\mc S^{\Gamma} (\mc A)$ by $\mc E^{\Gamma} (\varphi)$ and $\mc S^{\Gamma} (\varphi),$ respectively. In this scenario, our main  goal  is to provide a detailed  study of $\mc S^\Gamma(\mc A)$-subspace duals of a Bessel family/frame $\mc E^{\Gamma}(\mc A)$ in $L^2 (\mathscr G)$ due to its wide use   into the various areas like, harmonic analysis,  mathematical physics,    quantum mechanics,   quantum optics, etc.
Our results have so many predecessors related to the work on  subspace and  alternate duals,  orthogonal Bessel pair, etc.  
\cite{christensen2005generalized, christensen2004oblique, gumber2018orthogonality, gumber2019pairwisebuletin,  heil2009duals, gabardo2007uniqueness, bownik2019multiplication,kim2007pair,weber2004orthogonal}.  
The purpose of this paper is devoted to characterize a pair of orthogonal frames, and  subspace dual   of a Bessel family/frame  generated by  the $\Gamma$-TG system $\mc E^\Gamma(\mc A)$ in $L^2(\G)$ (see Theorems \ref{T:dualityMultiGenerators},\ref{T:Orthogonal}, \ref {T:orthogonalsupplementary}  \ref{Th:dual existence}, and \ref{Th: Dual-Ortho-Non Discerete}). When $\mc E^\Gamma(\mc A)$ is a Riesz basis then there is an associated biorthogonal system, which forms a unique dual  and it is called \textit{biorthogonal dual}. A brief study of biorthogonal system with discrete translation is  discussed here (see Theorem  \ref{T:Bio2}). We characterize such results using the Zak transform $\mc Z$ for the pair $(\mathscr G, \Gamma)$ defined by (\ref{eq:ZakTransform}). For the case of locally compact abelian group  $\mc G,$  we use the fiberization map $\mathscr T$  which unifies the classical results related to the orthogonality and duals of a Bessel family/frame associated with a TI  space (see Theorem \ref{Th: Dual-Ortho-Non Discerete}). This study of frames for their  orthogonality also enables us to discuss dual for the super Hilbert space $\oplus^{N}L^2(\G)$ (see Theorem \ref{T:Superdual}). As an application we have shown our results for Gabor systems (see Theorems \ref{T:Gaborcountable}, and \ref{TMI-Alternate Dual Thm}), splines,  the locally compact field of $p$-adic numbers $\mathbb Q_p$, and semidirect product of LCA groups, etc.

These duals are typically studied in the Euclidean context. Christensen and Eldar provided characterizations of  oblique duals for singly generated systems by the action of integers in \cite{christensen2004oblique}. Later they generalized it for multi-generators in \cite{christensen2005generalized}. Additionally, Gabardo and Hemmat investigated and characterized such duals and their uniqueness in \cite{gabardo2007uniqueness} for integer translation generated systems.
Heil, Koo, and Lim expanded the concept of various types of duals for the separable Hilbert spaces \cite{heil2009duals}. At the same time, the concept of subspace dual was started by Weber \cite[Definition 1.2]{weber2004orthogonal}. In this line of inquiry, our primary objective is to conduct a comprehensive study of subspace duals in the discrete and continuous setup of locally compact groups, with the hopes of unifying the previous studies. Orthogonality is a key concept to discuss the duals for super spaces \cite{han2000frames}. Kim et al. in  \cite{kim2007pair} established remarkable development following Weber \cite{weber2004orthogonal}. Here we also provide a detailed study of subspace orthogonality in the most general setup of locally compact group.

 One of the advantages of doing research on the pair $(\mathscr G, \Gamma)$ using the Zak transform is that it allows researchers to access a large number of pairs that were inaccessible before (using the Fourier Transform). These pairs include things like $(\mathbb R^n, \mathbb Z^m),$      $(\mathbb R^n,  \mathbb R^m),$ $(\mathbb Q_p, \mathbb Z_p),$ $(\mathcal G, \Lambda),$ etc., where  $n\geq m,$  $\Lambda$ (not necessarily co-compact, i.e., $\mathcal G/\Lambda$-compact, or uniform lattice) is a  closed subgroup of  the second countable  LCA group $\mathcal G,$ and   $\mathbb Z_p$ is the  $p$-adic integers in the $p$-adic number $\mathbb Q_p.$ This research goes beyond the boundary line of the locally compact abelian (LCA) group  with discrete translations, and now it includes the non-abelian groups with continuous translations. Furthermore, in the past, the application of the Zak transform was restricted to the case of  Gabor systems; however, we now apply it for translation generated systems.


\subsection{Organization}
The structure of the paper is as   follows. Few preliminaries about  subspace frames and  bracket map for locally compact group are discussed in Section \ref{Pre}. By action of a closed discrete abelian subgroup $\Gamma$  of $\G$ on a sequence of functions $\mc A$ in $L^2(\G),$  we study $\mc S^\Gamma(\mc A)$-subspace orthogonal and duals to a Bessel family/frame  $\mc E^{\Gamma}(\mc A)$ in Section \ref{s:discrete-translation} \,  and obtain characterization results in terms of the Zak transform for the pair $(\G, \Gamma)$ and Gramian operator.  As an application,  we study $\mc E^{\Gamma}(\mc A)$  for singly generated system in    Subsection \ref{ss:disc-singlegen}  .         An investigation of a translation generated biorthogonal system and dual of a Riesz basis is    carried out in   Section   $\ref{s:biortho}.$  In Section \ref{s:non-discrete}, we address the theory for a collection of generators indexed by $\sigma$-finite measure space (need not be countable) by the action of any closed abelian subgroup $\Gamma$ of $\G$ which unifies the broader class of continuous frames as well.  This paper ends with an illustration of our results for the various potential applications,  like,   Splines, Gabor systems,   $p$-adic fields $\mathbb Q_p,$  and semidirect product of LCA groups, etc., in Section \ref{App}.

\section{Preliminary}\label{Pre}
\subsection{Frames for subspaces} 
Frame generalizes the concepts of basis in a Hilbert space and helps to expand any vectors  in the series expansion in terms of frame elements \cite{aldroubi2017dynamical,bownik2000structure,bownik2007SMI, deboor1994structure}. Noways the  usual concept of frames has been extended for the case of frame for subspaces  which helps to include the series representation of any elements for subspaces also (see  \cite{christensen2005generalized,christensen2004oblique,jakobsen2016co}). Dual frames are the usual tools for such series representations while expansions \cite{christensen2004oblique}. In this article, we will discuss frames for subspaces and its associated duals in a Hilbert space where it is possible that the elements of the frame are not also elements of the subspace. Orthogonality of frames is a key concept to discuss duals for super Hilbert spaces which has enormous uses in multiple access communications and perfect reconstructions \cite{balan2000multiplexing,bhatt2007orthogonal,weber2004orthogonal}. The detailed study of orthogonal frames for the Euclidean setup  and locally compact abelian (LCA) group setup has been found in \cite{gumber2018orthogonality,gumber2019pairwisebuletin,kim2007pair,weber2004orthogonal}.

 Let $\mc H$ be a complex Hilbert space, and let $(\mathscr M,\sum_\mathscr M,\mu_\mathscr M)$ be a measure space, where $\Sigma_\mathscr M$ denotes $\sigma$-algebra and $\mu_{\mathscr M}$ the non-negative measure.
\begin{defn}
A family of vectors $\{f_k\}_{k\in \mathscr M}$ in $\mc H$ is called  a \textit{continuous $\mc K$-subspace frame}  (simply call as \textit{$\mc K$-subspace frame} \cite{weber2004orthogonal}, or \textit{basic frame} \cite{jakobsen2016co})   with respect to $(\mathscr  M,\sum_\mathscr M,\mu_\mathscr M)$ if
 $k \mapsto f_k$ is weakly measurable, i.e.,  the map $k \mapsto \langle f,f_k\rangle$ is measurable for each $ f\in \mc H$, and
 there exist constants $0<A\leq B<\infty$  such that
 \begin{equation}\label{framedefn}
 A\|f\|^2\leq\int_{\mathscr M} |\langle f,f_k\rangle |^2\  d{\mu_{\mathscr M}} (k)\leq B\|f\|^2 \quad \text{for all}
 \  f\in \mc K.
 \end{equation}
\end{defn}
If $\mc K=\mc H,$   then       $\{f_k\}_{k\in \mathscr M}$ is  a  \textit{frame for $\mc H$} (or \textit{total frame} \cite{jakobsen2016co}), and it is \textit{Bessel} in $\mc H$   when  only upper bound holds in (\ref{framedefn}),  and \textit{complete} when $\ol{\Span}  \{f_k\}_{k\in \mathscr M}=\mc H.$ 
\begin{defn}
For two Bessel families  $\{f_k\}_{k\in \mathscr M}$  and $\{g_k\}_{k\in \mathscr M}$   in $\mc H,$ if $\{g_k\}_{k\in \mathscr M}$ satisfies the following reproducing formula: 
\begin{equation}\label{subspace dual}
\int_{\mathscr M}  \langle f, g_{k} \rangle f_k\ d{\mu_{\msc M}} (k)    = f \  \text{for all}
\  f\in \mc K, 
\end{equation}	
then $\{g_{k}\}_{k\in \msc M}$ is called a \textit{$\mc K$-subspace dual} to   $\{f_k\}_{k\in \mathscr M}.$
\end{defn}
  Note that $\{f_k\}_{k\in \mathscr M}$ need not be a $\mc K$-subspace dual to $\{g_k\}_{k\in \mathscr M},$ but the family $\{g_{k}+h_{k}\}_{k\in \mathscr M}$ is    a $\mc K$-subspace dual to $\{f_k\}_{k\in \mathscr M}$ provided  
\begin{equation}\label{orthogonal}
\int_{\msc M}  \langle f, h_{k} \rangle g_k  \ d{\mu_{\msc M}} (k)   = 0 \  \text{for all}
\  f\in \mc K, 
\end{equation}	where $\{h_{k}\}_{k\in \msc M}$ is Bessel in $\mc H.$ Such $\{h_{k}\}_{k\in \msc M}$ satisfying (\ref{orthogonal}) is known as \textit{$\mc K$-subspace orthogonal} to  $\{f_k\}_{k\in \mathscr M} $  \cite{weber2004orthogonal}.

Every frame or a Bessel family is associated with an analysis operator, the range of which carries out a lot of information of a signal/image or function. Given a Bessel family $\mathcal X=\{f_k\}_{k \in \mathscr M}$  in $\mc H$  we define a bounded linear operator $T_{\mc X} :\H \ra L^2(\msc M) ,$ known as \textit{analysis operator},   by
\begin{equation}\label{AnalysisOp}
T_{\mc X}  (f)(k)=\langle f, f_k\rangle \   \mbox{for all}   \ k\in \msc M \ \mbox{and} \ f \in \H, 
\end{equation}
and  its adjoint operator $T_{\mc X}^*: L^2(\msc M)\ra \H,$ known as \textit{synthesis operator}, by
\begin{equation}\label{SynthesisOp}
  T_{\mc X}^* \psi  = \int_{\msc M}  \psi(k)  f_k  \  d{\mu_{\msc M}} (k)  \   \mbox{for all}   \  \psi \in L^2(\msc M),
\end{equation}
in the    weak sense. For   two Bessel families  $\mathcal X=\{f_k\}_{k \in \mathscr M}$ and $\mc Y:=\{g_{k}\}_{k\in \mathscr M}$ in $\mc H$ the operator $T_{\mc X}^* T_{\mc Y }: \mc H\ra \mc H$ given by $f \mapsto   \int_{\mc M} \langle f, g_{k} \rangle f_k  \ d{\mu_{\mc M}(k)}$ is known as \textit{mixed dual Gramian}. 
 \begin{defn}
Let $\mathcal X=\{f_k\}_{k \in \mathscr M}$ and $\mc Y:=\{g_{k}\}_{k\in \mathscr M}$ be two Bessel families  in $\mc H.$  \begin{enumerate}
\item[(i)] If  $T_{\mc X}^* T_{\mc Y } =I_{\ol{\Span}  \mc X},$ and  $\mc X,$ $\mc Y$ are the frames   for $\ol{\Span}  \mc X=\ol{\Span}  \mc Y \subseteq\mc H,$ then  $\mc X,$ $\mc Y$ are called as \textit{dual frames} to each other, where $I_{\ol{\Span}  \mc X}$ is an identity operator on $\ol{\Span}  \mc X.$
\item[(ii)] If    $T_{\mc X}^* T_{\mc Y }= 0$ on $\ol{\Span}  \mc X=\ol{\Span} \mc Y,$  then we call $\mc X$ and $\mc Y$ are     \textit{orthogonal Bessel  pair}.
\end{enumerate}    
\end{defn}
Orthogonality  plays a prominent role  to generalize dual frames for super Hilbert spaces \cite{weber2004orthogonal}. For more details on  orthogonal Bessel families, we refer \cite{gumber2018orthogonality, gumber2019pairwisebuletin,bownik2019multiplication,kim2007pair,weber2004orthogonal}. Next for the case of $\mc K$-subspace dual (\ref{subspace dual}) and $\mc K$-subspace orthogonal (\ref{orthogonal}) to  $\mathcal X=\{f_k\}_{k \in \mathscr M}$ we can write  $T_{\mc X}^*T_{\mc Y}\big|_{\mc K} = I_{\mc K}$ and $T_{\mc X}^*T_{\mc Y}\big|_{\mc K}=0,$ respectively. In particular,    if  $\mc X$ is a $\mc K$-subspace frame for  $\mc K= \ol{\Span}\{f_k\}_{k \in \msc M} ,$    and $T_{\mc X}^*T_{\mc Y}\big|_{\mc K} = I_{\mc K},$ then $\mc Y$ is    an   \textit{alternate dual}  to  $\mc X.$ We refer   \cite{christensen2004oblique, heil2009duals,  gabardo2007uniqueness}   for more details on alternate duals.  

Our aim is to discuss the aforementioned subspace dual and orthogonal frames for translation generated systems using the Zak transform and bracket map.
\subsection{Zak transform and Bracket}

 The \textit{Zak transform} $\mathcal Z$ of  $f\in L^1(\mathscr G)\cap L^2(\mathscr G)$ for the pair $(\mathscr G,\Gamma)$  is   defined by 
\begin{equation}\label{eq:ZakTransform}
(\mc Zf)(\alpha, \Gamma x) =\widehat{f^{\Gamma x}}(\alpha)=\int_{\Gamma} f^{ \Gamma x} (\gamma) \alpha (\gamma^{-1})\ \dalpha \  \mbox{a.e.} \  \alpha \in \widehat{\Gamma} \ \mbox{and} \   \Gamma x\in \Gamma\backslash \mathscr G,
\end{equation}
which is a    unitary linear transformation from $ L^2(\mathscr G)  $ to $L^2(\widehat{\Gamma}\times \Gamma\backslash \mathscr G)$ \cite{iverson2015subspaces}. The space $L^2(\widehat{\Gamma}\times \Gamma\backslash \mathscr G)$ is measure theoretic isomorphic to $L^2(\widehat{\Gamma}; L^2(\Gamma\backslash \mathscr G)).$ For $x \in  \mathscr G,$  $\Gamma x$ is a \textit{right coset} of $\Gamma$ in $\mathscr G$ with respect to $x.$ The   set        $\Gamma\backslash \mathscr G = \{\Gamma x: x \in  \mathscr G\}$ denotes    a collection of all right cosets of $\Gamma $ in $\mathscr G.$  
 For  a function   $f :\mathscr G \rightarrow \mathbb C,$     a  complex valued    function $f^{\Gamma x}$ on $\Gamma$ is given  by  
\begin{align}\label{borelsection}
f^{ \Gamma x}(\gamma)=f(\gamma \, \Xi(\Gamma x)), \quad     \gamma \in  \Gamma, 
\end{align}
where   $\Xi : \Gamma \backslash \mathscr G\ra  \mathscr G$ is a \textit{Borel section} for  the quotient space $\Gamma \backslash \mathscr G$ whose   existence    is  guaranteed by   \cite[Lemma 1.1]{Mackey1952InducedRepresentations}. The \textit{dual group} $\widehat{\Gamma}$ of $\Gamma$ is a collection of continuous homomorphisms from $\Gamma$ to the circle group $\mathbb T.$ 
In case of  locally compact abelian (LCA) group $\G,$ the  Zak transform $\mc Z$ has a sibling named as the \textit{fiberization map} $\mathscr T,$ 
\begin{align}\label{fiberizationmap}
\mathscr T:&L^2(\G)\ra L^2(\widehat {\G}/\Gamma^\perp; L^2(\Gamma^\perp))\ \mbox{ defined by}\\ \nonumber
 & (\mathscr Tf)(\omega \Gamma^\perp)(k)=\widehat{f}(\Theta(w\Gamma^\perp)k)  
\end{align}
is  unitary, where $f \in L^2(\G)$ $w\Gamma^\perp \in \widehat {\G}/\Gamma^\perp,$ and $k \in \Gamma^\perp.$
 The group $\Gamma^\perp=\{\omega\in \widehat \G:\omega (\gamma)=1\ \mbox{ for all} \  \gamma \in \Gamma\}$ and the quotient group $\G/\Gamma$ is same as $\Gamma\backslash \mathscr G.$ The Borel section $\Theta: \widehat {\G}/\Gamma^\perp \ra \widehat {\G}$ sends compact sets to pre-compact sets.  
 In the setup of discrete group $\Gamma,$ the Zak transform $\mc Z$ for the pair $(\G, \Gamma)$ can be rewritten  from (\ref{eq:ZakTransform})  as follows: for $f\in L^1(\G) \cap L^2(\G),$ $ (\mc Z f )(\alpha, \Gamma x)=\widehat{f^{ \Gamma x}} (\alpha)=\sum_{\gamma\in\Gamma} f^{ \Gamma x} (\gamma) \alpha (\gamma^{-1})   \ \mbox{for} \  \alpha \in \widehat{\Gamma} \ \mbox{and} \ \Gamma x \in \Gamma \backslash \G. $
 We begin with the notion of matrix elements for the left regular representation   \cite[Section 5.2]{folland2016course}.
	\begin{defn}
	For $\varphi, \psi \in L^2(\mathscr G),$ let  $\mc M_\varphi \psi:\Gamma\ra \mathbb C$ be   a function defined by 
	 		\begin{align}\label{Matrixelement}
	(\mc M_\varphi \psi)(\gamma) =\left\langle \psi, L_\gamma \varphi\right\rangle,    \  \gamma \in \Gamma.\nonumber
	\end{align}  
	Then $\mc M_\varphi \psi$ is known as a  \textit{matrix element} of the left regular representation associated with  $\varphi$ and $\psi.$
	\end{defn}

	In the sequel, we require the discrete-time Fourier transform of  $\mc M_\varphi \psi$  at a point of $\widehat \Gamma.$  	Recall  a \textit{discrete-time Fourier transform} $\widehat {z} (\alpha)$ of a sequence $z=(z(\gamma))\in \ell^2(\Gamma)$ at a point $\alpha \in \widehat \Gamma,$ defined by    $ 	\widehat{z}(\alpha)=\sum_{\gamma\in \Gamma} z(\gamma)\ol{\alpha(\gamma)}. $   The convergence of the series is interpreted as its limit in $L^2(\widehat \Gamma).$ We describe below the discrete-time Fourier transform of  $\mc M_\varphi \psi$ in terms of the Zak transform associated with the right coset $\Gamma \backslash \G.$

	\begin{lem}\label{L:matrixelement}
		 Let $\varphi, \psi \in L^2(\mathscr G)$ be such that the associated matrix element  $\mc M_\varphi \psi$ is a member of $\ell^2 (\Gamma).$ Then the discrete-time Fourier transform of   $\mc M_\varphi \psi$    at $\alpha \in \widehat \Gamma$  is 
		$$
		\widehat{(\mc M_\varphi \psi)}(\alpha)= [ \psi,  \varphi](\alpha),
		$$ 
		provided $[\psi,  \varphi](\cdot)\in L^2 (\widehat{\Gamma}),$  
		where    the complex valued function $[ \psi, \varphi](\cdot)$ on $\widehat \Gamma,$ known  as bracket map,   is given by
		\begin{equation}\label{eq:bracket}
		[ \psi,  \varphi](\alpha):=\int_{\Gamma \backslash \mathscr G}   \mc Z\psi(\alpha, \Gamma x)\, \overline{\mc Z\varphi(\alpha, \Gamma x)}\ d{\mu_{\Gamma \backslash \mathscr G}}(\Gamma x) \ \mbox{for} \  \alpha \in \widehat \Gamma. 
		\end{equation} 
	
		Moreover,   for a Bessel sequence $\mc E^{\Gamma} (\varphi)=\{L_\gamma\varphi: \gamma\in \Gamma\}$ in $L^2 (\G),$ 	$\mc M_\varphi \psi$ and $[ \psi,  \varphi](\cdot)$  are  members of $\ell^2 (\Gamma)$ and   $L^2(\widehat \Gamma),$ respectively, and hence the above result holds true.  
		\end{lem}
	\begin{proof}  Since $\mc Z$ is unitary, the discrete-time  Fourier transform of $\{\mc M_\varphi \psi (\gamma)\}_{\gamma \in \Gamma}\in \ell^2 (\Gamma)$ at $\alpha \in \widehat \Gamma$ is, 
			\begin{align*}
			\widehat{(\mc M_\varphi \psi)}(\alpha)= &\sum_{\gamma\in \Gamma} \langle \psi, L_{\gamma} \varphi \rangle \alpha(\gamma^{-1})  =
			\sum_{\gamma\in \Gamma}\langle\mc Z\psi,\mc Z(L_{\gamma} \varphi)\rangle_{L^2(\hat{\Gamma} \times   \Gamma\backslash \G)} \  \alpha(\gamma^{-1}).
			\end{align*}
		 Employing  the intertwining property of the Zak transform $\mc Z$ on the  left translation with modulation, i.e.,  for  $f \in L^1 (\G) \cap L^2(\G),$ $\gamma \in \Gamma,$ $\alpha \in \widehat{\Gamma},$ $x \in \G,$  and  using		   (\ref{eq:ZakTransform}) and (\ref{borelsection}), 
		\begin{align*}
		\mc Z (L_\gamma f)(\alpha, \Gamma x)= &  \sum_{\lambda \in \Gamma} (L_\gamma f)^{\Gamma x} (\lambda) \alpha (\lambda^{-1}) =  \sum_{\lambda \in \Gamma} f ((\gamma^{-1} \lambda) \Xi (\Gamma x)) \alpha (\lambda^{-1})\\
		= &  \sum_{\lambda \in \Gamma}   f^{\Gamma x} (\gamma^{-1} \lambda) \alpha (\lambda^{-1})  = \alpha(\gamma^{-1})\mc Zf(\alpha, \Gamma x).
		\end{align*}
   Then  we obtain  
			\begin{align}\label{eq:DFT-matrix}
			\widehat{(\mc M_\varphi \psi)}(\alpha) =  \sum_{\gamma\in \Gamma} \alpha(\gamma^{-1}) \int_{\widehat{\Gamma}}  \beta(\gamma) \langle\mc Z\psi(\beta), \mc Z\varphi(\beta)\rangle_{L^2(\Gamma\backslash \G)}\ d{\mu_{\widehat \Gamma}}(\beta)
			=  \sum_{\gamma\in \Gamma} \alpha(\gamma^{-1}) \zeta(\gamma), 
			\end{align}
where for $\gamma \in \Gamma,$ the function $\zeta(\gamma):=\int_{\widehat\Gamma}    {\beta(\gamma)} \langle \mc Z\psi(\beta),\mc Z\varphi(\beta)\rangle\ d{\mu_{\widehat \Gamma}}(\beta)$ is identical with $\mc M_\varphi \psi (\gamma).$
			The sequence  $\{\zeta(\gamma)\}_{\gamma \in \Gamma} \in \ell^2 (\Gamma)$ since  $\{\mc M_\varphi \psi (\gamma)\}_{\gamma \in \Gamma}\in \ell^2 (\Gamma).$ 
		Further, we can write (\ref{eq:DFT-matrix}) for $\alpha \in \widehat \Gamma$ as follows:
		  		\begin{align*}
		\widehat{(\mc M_{\varphi} \psi)}(\alpha) &=  \sum_{\gamma\in \Gamma} \alpha(\gamma^{-1}) \zeta(\gamma) 
		 =\sum_{\gamma\in \Gamma} \left( \int_{\widehat\Gamma}    {\beta(\gamma)} \langle \mc Z\psi(\beta),\mc Z\varphi(\beta) \rangle\ d{\mu_{\widehat \Gamma}}(\beta) \right)\alpha(\gamma^{-1})\\
		  &= \sum_{\gamma\in \Gamma} \left( \int_{\widehat\Gamma}    {\beta(\gamma)} \left( \int_{\Gamma\backslash \G} \mc Z\psi(\beta, \Gamma x) \overline{\mc Z\varphi(\beta, \Gamma x)}\ d{\mu_{\Gamma\backslash \G}}(\Gamma x)	\right)	 	 d{\mu_{\widehat \Gamma}}(\beta) \right)\alpha(\gamma^{-1})\\
		  &= \sum_{\gamma\in \Gamma}  \left(\int_{\widehat\Gamma}      [ \psi,   \varphi](\beta)  {\beta(\gamma)}\	 	 d{\mu_{\widehat \Gamma}} (\beta) \right)\alpha(\gamma^{-1}), 
		 	\end{align*}
		 where 	for $\beta \in \widehat{\Gamma},$ $[\psi,  \varphi](\beta)$ is defined by 
		 $ 		 [\psi,   \varphi](\beta)=\int_{\Gamma\backslash \G} \mc Z\psi(\beta, \Gamma x) \overline{\mc Z\varphi(\beta, \Gamma x)}\ d{\mu_{\Gamma\backslash \G}}(\Gamma x).$
		Also  by identifying  $\Gamma$ to $\widehat{\widehat{\Gamma}}$ as  $\gamma \mapsto \widehat{\gamma}$ and noting that $\widehat{\widehat{\Gamma}}$ is an orthonormal basis for $L^2 (\widehat{\Gamma}),$  we can write 
		 		\begin{align*}
		 		\widehat{(\mc M_{\varphi}\psi)}(\alpha)  =&  \sum_{\gamma\in \Gamma} \left( \int_{\widehat\Gamma}    [\psi,   \varphi](\beta)  \widehat{\gamma}(\beta)\ d{\mu_{\widehat \Gamma}}(\beta) \right) \overline{\widehat{\gamma}(\alpha)} 
		 		 =  \sum_{\gamma\in \Gamma} \left\langle    [\psi,   \varphi](\cdot),  \overline{\widehat{\gamma}}(\cdot)\right \rangle  \overline{\widehat{\gamma}(\alpha)}
		 		 =   [\psi,   \varphi](\alpha),
		 		\end{align*}
		 provided		 $ [\psi,  \varphi](\cdot)\in L^2 (\widehat{\Gamma}).$  	
		
	For the moreover part, assume that the $\Gamma$-TG system $\mc E^\Gamma (\varphi)$	 is Bessel. Then for all $f\in L^2 (\G),$ we have the inequality $\sum_{\gamma \in \Gamma} |<f, L_\gamma \varphi>|^2 \leq B \|f\|^2$
	 for some constant $B>0,$ and hence by choosing $f=\psi,$ we get $\mc M_{\varphi}\psi$ as a member of $\ell^2 (\Gamma).$ Also the Bessel  property of   $\mc E^\Gamma (\varphi)$	   implies   $[ \varphi,  \varphi](\alpha) \leq B$   a.e. $\alpha \in \widehat \Gamma$   \cite{iverson2015subspaces},   and hence  	  using the Cauchy-Schwarz inequality
	 \begin{align*}
	 \int_{\widehat \Gamma}\Big|[\psi, \varphi](\alpha)\Big|^2\ \dalpha = & \int_{\widehat \Gamma} \Big|\int_{\Gamma \backslash \G}\mc Z\psi (\alpha, \Gamma x)\ol{\mc Z\varphi(\alpha, \Gamma x)}\ \dGammax\Big|^2  \dalpha\\
	 \leq & \int_{\widehat \Gamma}  \left(\int_{ \Gamma \backslash \mathscr G}\Big|\mc 
	 Z \psi(\alpha, \Gamma x)\Big|^2\ \dGammax\right) 
	  \left(\int_{ \Gamma \backslash \mathscr G}\Big|\mc 
	 Z \varphi(\alpha, \Gamma x)\Big|^2\ \dGammax\right)\dalpha\\
	 	=& \int_{\widehat \Gamma}  [ \psi,  \psi](\alpha) [ \varphi,  \varphi](\alpha)\ \dalpha\\
	 	 \leq &  B \int_{\widehat \Gamma}  [  \psi,  \psi](\alpha) \ \dalpha= B \|\mc Z \psi\|^2 
	 	=  B \|\psi\|^2.
	 \end{align*}
	  \end{proof}
Also note  that for $\varphi, \psi\in L^2(\msc G), [\varphi, \psi](.)\in L^1(\widehat \Gamma) .$
	When the locally compact group $\G$ becomes abelian, \textit{denoted by} $\mc G,$ the groups   	$\widehat {\mc G}/\Lambda^\perp $ and $\widehat{\mc G/\Lambda}$ are topologically isomorphic to $\widehat \Lambda$  and $\Lambda^\perp,$ respectively \cite{folland2016course},  where  $\Lambda$ is a closed discrete subgroup of $\mc G.$   Instead of the Zak transform $\mc Z$ for the pair $(\G, \Gamma),$ we will use the fiberization map (\ref{fiberizationmap})  for the pair   $(\mc G, \Lambda).$ 
	  
	   For any   $\varphi, \psi \in L^2(\mc G),$ \textit{the bracket map using the fiberization will be denoted by}  $[\psi, \varphi]_{\mathscr T} (\cdot)$ to distinguished it from the earlier notation $[\varphi,\psi](\cdot)$ using the Zak transform. Next we obtain  a relation  between $[\psi, \varphi] (\cdot)$ and $[\psi,  \varphi]_{\mathscr T}(\cdot)$ in the setup of locally compact abelian (LCA) group  $\mc G$ and its closed discrete subgroup $\Lambda,$ where 
		\begin{equation}\label{eq:bracket-abelain}
	[\psi, \varphi]_{\mathscr T}(\omega \Lambda^\perp):=\int_{\Lambda^\perp}  \mathscr T\psi(\omega\Lambda^\perp)(\xi)\ol{\mathscr T\varphi(\omega\Lambda^\perp)(\xi)}\ d{\mu_{\Lambda^\perp}}(\xi) \   \ \mbox{for} \  \omega \Lambda^\perp \in \widehat{\mc G}/\Lambda^\perp, 
	\end{equation} 
	which is  a some kind of reminiscence of \cite{iverson2015subspaces}. 
	Since 	$[\psi, \varphi] (\cdot) \in L^1 (\widehat\Lambda)$    the Fourier transform $\mc F$ of $[\psi,\varphi] (\cdot)$ at $\lambda \in \Lambda$  can be written as  
	 	\begin{align*}
	 		 	\mc F[\psi, \varphi](\lambda)&=\int_{\widehat \Lambda}[\psi,\varphi](\beta) \ol{\beta(\lambda)}\ d{\mu_{\widehat \Lambda}}(\beta)=\int_{\widehat \Lambda}\left(\int_{\Lambda\backslash\mc G}\mc Z\psi(\beta, \Lambda x)\ol{\mc Z\varphi(\beta, \Lambda x) }\ d{\mu_{\Lambda\backslash \mc G}}\right)\ol{{\beta}(\lambda)}d{\mu_{\widehat \Lambda}}(\beta)\\
	 		 	&=\int_{\widehat \Lambda}\langle \mc Z\psi(\beta),\mc Z\varphi(\beta)\rangle\ol{{\beta}(\lambda)}\ d{\mu_{\widehat \Lambda}}(\beta)=\int_{\widehat\Lambda}\langle \mc Z(L_{\lambda}\psi)(\beta),\mc Z\varphi(\beta)\rangle\  d{\mu_{\widehat \Lambda}}(\beta)\\
	 		 	&=\langle L_{\lambda}\psi,  \varphi\rangle.
	 	\end{align*}
	 Since 	$[\psi, \varphi]_{\mathscr T} (\cdot) \in L^1 (\widehat{\mc G}/\Lambda^\perp),$ from the similar calculations of    Lemma \ref{L:matrixelement},  and the  groups $\widehat {\mc G}/\Lambda^\perp$ and 	$\widehat \Lambda$ are topologically  isomorphic,     the  Fourier transform $\mc F$ of $[ \psi, \varphi]_{\mathscr T} (\cdot)$ at $\lambda \in \Lambda$  can be written as  follows:
		\begin{align*}
	 		\mc F[\psi,\varphi]_{\mathscr T}(\lambda) & =\int_{\widehat \Lambda}[\psi, \varphi]_{\mathscr T}(\omega \Lambda^\perp)\ol{\omega(\lambda)}\ d{\mu_{\widehat \Lambda}}(\omega|_{\Lambda})
	 		=\int_{\widehat \Lambda} \left(\int_{\Lambda^\perp}  \mathscr T\psi(\omega\Lambda^\perp)(\xi)\ol{\mathscr T\varphi(\omega\Lambda^\perp)(\xi)}d{\mu_{\Lambda^\perp}}(\xi) \right)\ol{\omega(\lambda)}\ d{\mu_{\widehat \Lambda}}(\omega|_{\Lambda}) \  \\
	 			&=\int_{\widehat \Lambda}\langle \mathscr T\psi(\omega\Lambda^\perp), \mathscr T\varphi(\omega\Lambda^\perp)\rangle\ol{\omega(\lambda)}\ d{\mu_{\widehat \Lambda}}(\omega|_{\Lambda}) 
	 			=\int_{\widehat {\mc G}/\Lambda^\perp}\langle \mathscr T\psi(\omega\Lambda^\perp), \mathscr T\varphi(\omega\Lambda^\perp)\rangle\ol{\omega(\lambda)}\ d{\mu_{\widehat {\mc G}/\Lambda^\perp}}(w\Lambda^\perp).
 			\end{align*} 
	 		Employing the unitary property of the fiberization map $\mathscr T,$ 
	 	we have 	
	 	\begin{align*}			
	 	\mc F[\psi, \varphi]_{\mathscr T}(\lambda) 	&	=\int_{\widehat {\mc G}/\Lambda^\perp}\langle \mathscr T L_{\lambda}\psi(\omega\Lambda^\perp),  \mathscr T \varphi(\omega\Lambda^\perp)\rangle\  d{\mu_{\widehat{\mc G}/\Lambda^\perp}}(\omega\Lambda^\perp)=\langle L_{\lambda} \psi, \varphi\rangle,
	 	\end{align*}
	 	since    the fiberization map $\mathscr T$  intertwines   left translation with modulation, i.e., 
	 \begin{align*}
	 \mathscr TL_\lambda f(\omega \Lambda^\perp)(\xi)=&\widehat{(L_\lambda f)}(\Theta(\omega\Lambda^\perp)\xi)=\Theta(\omega\Lambda^\perp)(\lambda^{-1}) \xi(\lambda^{-1})\mathscr Tf(\omega\Lambda^\perp)(\xi)
	 =\omega(\lambda^{-1})\mathscr Tf(\omega\Lambda^\perp)(\xi),
	 	\end{align*}
	 	as Fourier transform intertwines left translation with modulation and $\Theta: \widehat{\mc G}/\Lambda{^\perp}\ra \widehat{\mc G} $ is a Borel section $\Theta(\omega\Lambda^\perp)=\omega\eta$  for some $\eta\in \Lambda^{\perp}$ and $\eta(\lambda^{-1})=\xi(\lambda^{-1})=1.$
	Therefore for all $\lambda \in \Lambda,$ we have 	$\mc F[\psi, \varphi](\lambda)=	\mc F[\psi, \varphi]_{\mathscr T}(\lambda)$ which implies 
\begin{align}\label{zak-fib}
	[\psi, \varphi](\omega|_\Lambda)=[\psi, \varphi]_{\mathscr T}(\omega \Lambda^\perp) \ \a.e. \ \omega \in \widehat {\mc G} 
\end{align}
	 since  the Fourier transform $\mc F:L^1(\widehat\Lambda)\ra C_{0}(\Lambda)$ is   injective.   Thus by using the relation (\ref{zak-fib}),  we state the following result analogous to Lemma \ref{L:matrixelement}   for the case of an LCA group $\mc G$ and its closed discrete subgroup $\Lambda$ in terms of fiberization. In particular  the same result can be realized for the case of uniform lattice $\Lambda.$ By a \textit{uniform lattice} $\Lambda,$ we mean it is a closed discrete subgroup of an LCA group $\mc G$ such that    $\mc G/ \Lambda$ is compact.

	 	   \begin{lem} \label{L:zak-fib}
	 	   	Let $\mc G$ be  a locally compact abelian group and $\Lambda$ be a closed discrete subgroup of  $\mc G.$ If  $\varphi, \psi \in L^2(\mc G)$  such that the   matrix element  $\mc M_\varphi \psi$ is a member of $\ell^2 (\Lambda),$ then the discrete-time Fourier transform of   $\mc M_\varphi \psi$   in terms of the fiberization  for the pair  abelian $(\mc G, \Lambda)$  is  
	 	   	$$
	 	   	\widehat{(\mc M_\varphi \psi)}(\omega|_\Lambda)= [\psi, \varphi]_{\mathscr T}(\omega \Lambda^\perp), \quad \omega \Lambda^\perp \in \widehat{\mc G}/\Lambda^\perp
	 	   	$$ 
	 	   	provided $[\psi, \varphi]_{\mathscr T}(\cdot)\in L^2 (\widehat{\mc G}/\Lambda^\perp),$  
	 	   	where    the complex valued function $[\psi, \varphi]_{\mathscr T}(\cdot)$ on $\widehat{\mc G}/\Lambda^\perp$  is given by (\ref{eq:bracket-abelain}).
	 	   	Moreover,   for a Bessel sequence $\mc E^{\Lambda} (\varphi)$ in $L^2 (\mc G),$ 	$\mc M_\varphi \psi$ and $[ \psi,  \varphi]_{\mathscr T}(\cdot)$  are  members of $\ell^2 (\Lambda)$ and   $L^2(\widehat{\mc G}/\Lambda^\perp),$ respectively, and hence the above result holds true.  
	 	   	 	   	\end{lem}

\section{Subspace  dual   of a frame  by the action of a discrete abelian subgroup}\label{s:discrete-translation}
 Throughout the section,  we assume    that $\Gamma$ is  a   closed discrete abelian subgroup of a   second countable    locally compact group $\mathscr G.$    In this section we study    $S^\Gamma(\mc A)$-subspace duals of a Bessel/frame sequence $\mc E^{\Gamma}(\mc A)$ in $L^2 (\mathscr G)$ in terms of 
the Zak transform for the pair $(\G, \Gamma).$  Such study on the pair $(\G, \Gamma)$ allows to access the various number of previously inaccessible pairs, like $(\mathbb R^n, \mathbb Z^m),$ $(\mathbb Z^n, \mathbb Z^m),$ $(\mathbb Z_N^n, \mathbb Z_N^m),$ etc., where $n \geq m$ and $\mathbb Z_N$ is a group modulo $N.$ 
  In the present  section,  we   discuss subspace dual and orthogonal  frames  for $\Gamma$-TI  spaces generated by a countable number of functions, $\mc A=\{\varphi_t:t\in \mc N\}$  in $L^2(\G),$  where $\mc N$ is the $\sigma$-finite measure space having counting measure. We refer \cite{bhatt2007orthogonal,christensen2005generalized, kim2007pair,  weber2004orthogonal} regarding  the   orthogonality and duality  results of a frame on the Euclidean spaces and LCA groups using the Fourier transform.

The following result will play an important role to  study   the duals of an $\Gamma$-TG system using the Zak transform. As an additional point of reference, it expresses the transition from the role of $\Gamma$ to $\widehat \Gamma.$

\begin{lem}\label{L:BesselCharecterization}
	Let $\varphi$ and $\psi$ be two functions in $L^2(\mathscr G)$ be such that  the corresponding $\Gamma$-TG  systems  defined as (\ref{TIsystem}),  $\mc  E^{\Gamma}(\varphi)=\{L_\gamma\varphi:\gamma\in \Gamma\}$ and  $\mc E^{\Gamma}(\psi)=\{L_\gamma \psi:\gamma\in \Gamma\}$  are Bessel.  Then  for all $f, g \in L^2(\G),$ we have 
	\begin{align*}
	\sum_{\gamma\in \Gamma}\left \langle f, L_{\gamma}\varphi\right \rangle\langle L_{\gamma}\psi,g\rangle 
	=\int_{\widehat {\Gamma}} [f,\varphi](\alpha) [\psi,g](\alpha) \ d{\mu_{\widehat \Gamma}}(\alpha).
	\end{align*}
	Moreover, 	when the pair $(\G, \Gamma)$ is an abelian pair $(\mc G, \Lambda)$   for $f, g \in L^2(\mc G)$ and a.e. $\omega\Lambda^\perp\in \widehat{\mc G}/\Lambda^{\perp},$   $$\sum_{\lambda\in\Lambda}\left \langle f, L_{\lambda}\varphi\right \rangle\langle L_{\lambda}\psi,g\rangle 
=\int_{\widehat{\mc G}/\Lambda^{\perp}} [\msc Tf, \msc T\varphi](\omega \Lambda^\perp) [\msc T\psi,\msc T g](\omega \Lambda^\perp) \ d{\mu_{\widehat{\mc G}/\Lambda^{\perp}}}(\omega \Lambda^\perp)$$ in terms of  the fiberization.  
\end{lem}
\begin{proof}
For all $f \in L^2  (\G)$ and 	from (\ref{eq:bracket})
	\begin{align*}
	\left \langle f, L_{\gamma}\varphi\right\rangle = &	\left \langle \mc Zf, \mc Z L_{\gamma}\varphi\right\rangle= \int_{\widehat \Gamma}\langle \mc Zf(\alpha),\mc  Z\varphi(\alpha)\rangle \alpha(\gamma)\ \dalpha
	=  \int_{\widehat \Gamma} [f, \varphi](\alpha) \alpha(\gamma) \  \dalpha.
	\end{align*}
 Since the $\Gamma$-TG system $\mc  E^{\Gamma}(\varphi)$ is Bessel, we have $[f, \varphi](\cdot)\in L^2 (\widehat \Gamma)$ from Lemma \ref{L:matrixelement}, and hence using  the inverse  Fourier transform $[f, \varphi]^{\vee} (\gamma)$  at $\gamma \in \Gamma,$ the above expression can be written as follows: 
	$$
	\left \langle f, L_{\gamma}\varphi\right\rangle =  [f,\varphi]^{\vee} (\gamma). 
	$$
	Similarly,  we have $\left \langle  L_{\gamma}\psi, g\right\rangle = [\psi,  g]^{\vee} (\gamma)$ for $\gamma \in \Gamma$ and $g \in L^2  (\G).$ Further, note that the sequences $\{[f,\varphi]^{\vee} (\gamma)\}_{\gamma \in \Gamma}$ and $\{[\psi,  g]^{\vee} (\gamma)\}_{\gamma \in \Gamma}$ are members of $\ell^2 (\Gamma)$ follows from Lemma \ref{L:matrixelement}. Hence the result follows  by observing the Parseval formula on $\ell^2 (\Gamma)$ in  the following calculation
	 	\begin{align*} 
	\sum_{\gamma\in \Gamma}\left \langle f, L_{\gamma}\varphi\right \rangle\langle L_{\gamma}\psi,g\rangle & = \sum_{\gamma \in \Gamma} \left([f,\varphi]^\vee(\gamma) \right)  \left([\psi, g]^\vee(\gamma)\right)
	 = \left\langle [f,\varphi]^\vee,[g,  \psi]^\vee\right\rangle_{\ell^2(\Gamma)}\\\nonumber
	 &=\left\langle [f, \varphi](\cdot),[g,  \psi](\cdot)\right\rangle_{L^2(\widehat\Gamma)}
	  =\int_{ \widehat \Gamma }[f,\varphi](\alpha)  [\psi,g](\alpha)\ \dalpha.
	\end{align*}

  The moreover part  follows from  the same argument as above by substituting the Zak transform $\mc Z$ for the pair $(\msc G, \Gamma)$ with the fiberization $\mathscr T$ for the pair $(\mc G, \Lambda)$ by making use of the  Lemma \ref{L:zak-fib}.
\end{proof}

The  next result connects   analysis and synthesis operators with the pre-Gramian operator in terms of the Zak transform for the pair $(\G, \Gamma).$  We recall      $\Gamma$-TG  system $\mc E^{\Gamma} (\mc A)=\{L_{\gamma}\varphi\}_{\gamma \in  \Gamma, \varphi \in \mc A} $ and its associated  $\Gamma$-TI space $\mc S^{\Gamma} (\mc A)=\ol{\Span}  \mc E^{\Gamma} (\mc A)$   from (\ref{TIsystem}) for a countable collection $\mc A=\{\varphi_t:t\in \mc N\}$  in $L^2(\G).$  If $\mc E^{\Gamma}(\mc A)$ is  a  Bessel family in $L^2(\G),$  then   from  (\ref{AnalysisOp}) the  associated analysis operator     	$T_{\mc E^{\Gamma}(\mc A)}:L^2(\G) \ra \ell^2(\Gamma\times \mc N)$ is defined by $ f\mapsto \left\{\langle f, L_{\gamma}\varphi_t\rangle\right\}_{\gamma \in \Gamma, t\in \mc N}$  and from (\ref{SynthesisOp})  the associated synthesis operator is 
$T_{\mc E^{\Gamma}(\mc A)}^*:\ell^2(\Gamma\times \mc N)\ra L^2(\G)$ defined by $\{h_t(\gamma)\}_{t\in \mc N, \gamma \in \Gamma}\mapsto \sum_{t\in \mc N}\sum_{ \gamma\in\Gamma} h_t(\gamma)L_{\gamma}\varphi_t.$ Since  $\mc E^\Gamma(\mc A)$ is Bessel in $L^2(\G),$  the system $\{\mc Z\varphi_t (\alpha)=\{\mc Z\varphi_t(\alpha,\Gamma x)\}_{\Gamma x\in \Gamma\backslash \mathscr G }\}_{t \in \mc N}$ is also Bessel in $L^2(\Gamma\backslash \G)$  for a.e.  $\alpha \in \widehat{\Gamma}$  \cite{iverson2015subspaces},  and hence the associated  \textit{pre-Gramian operator } $\mf J_{\mc A}(\alpha)$ corresponding to the Bessel system $\mc E^{\Gamma}(\mc A)$  is  defined by $\mf J_{\mc A}(\alpha):\ell^2(\mc N)\ra L^2(\Gamma\backslash \mathscr G),$  $\eta=\{\eta_t\}_{t\in \mc N}\mapsto \left\{\sum_{t\in \mc N}\eta_t\mc Z\varphi_t(\alpha, \Gamma x)\right\}_{\Gamma x\in \Gamma\backslash \mathscr G },$
which  is a well defined  bounded linear operator due  the Bessel system $\mc A (\alpha).$  
Further, we define its adjoint operator  
$\mf J_{\mc A}(\alpha)^*: L^2(\Gamma\backslash \mathscr G)\ra\ell^2(\mc N)$ by $\nu\mapsto \left\{\left\langle \nu,  \mc Z\varphi_t(\alpha)\right\rangle\right\}_{t\in \mc N}.$  
The \textit{Gramian operator}     $G_\mc A (\alpha)=\mf J_{\mc A}(\alpha)^*\mf J_{\mc A}(\alpha)$ from $\ell^2(\mc N)$ to $\ell^2(\mc N)$  is also bounded linear operator for a.e.   $\alpha \in \widehat \Gamma.$  For two Bessel systems $\mc E^{\Gamma}(\mc A)$ and $\mc E^{\Gamma}(\mc A'),$   the associated \textit{mixed dual-Gramian operator}  $\tilde G_{\mc A,\mc A'}(\alpha)=\mf J_{\mc A}(\alpha)\mf J_{\mc A'}(\alpha)^*:L^2(\Gamma\backslash \G)\ra L^2(\Gamma\backslash \G)$  is defined by  
$\langle \tilde G_{\mc A,\mc A'}(\alpha)v_1,v_2\rangle  =\sum_{ t\in \mc N}\langle v_1, \mc Z\psi_t(\alpha)\rangle \ol{\langle v_2 , \mc Z\varphi_t(\alpha) \rangle}$
for  a.e. $\alpha \in \widehat \Gamma,$ where $v_1, v_2 \in L^2(\Gamma\backslash \G)$ and $\mc A'=\{\psi_t:t\in \mc N\}\subset L^2(\G).$ This terminology and the following proposition can be deduced using the fiberization $\mathscr T$ map for the abelian pair $(\mc G, \Lambda).$ 
	 \begin{prop}\label{P:zakpre-gramian}
	Let $\mc A=\{\varphi_t\}_{t\in \mc N}$  and $\mc A'=\{\psi_t\}_{t\in \mc N}$ be two countable collections of functions in $L^2 (\mathscr G)$ such that  $\mc  E^{\Gamma}(\mc A)$  and $\mc  E^{\Gamma}(\mc A')$ are   Bessel. Then the following are true: 
	\begin{itemize}
		\item[(i)] For each $t\in \mc N$ and $f \in L^2(\G),$ the Fourier transform of   $(T_{\mc E^{\Gamma}(\mc A)}f)_t$  is given by 
			$$\widehat{(T_{\mc E^{\Gamma}(\mc A)}f)_t}(\alpha) = [ f,  \varphi_t] (\alpha), \ \mbox{and} \  \left\{\widehat{(T_{\mc E^{\Gamma}(\mc A)}f})_{t}
			(\alpha)\right\}_{t\in \mathcal N} = \mf J_{\mc A}(\alpha)^*\{(\mc Zf)(\alpha,\Gamma x)\}_{\Gamma x \in\Gamma \backslash \mathscr G}$$ 
for a.e. $\alpha \in \widehat \Gamma,$ 	where $(T_{\mc E^{\Gamma}(\mc A)}f)_t=\left\{\langle f, L_{\gamma}\varphi_t\rangle\right\}_{\gamma \in \Gamma}.$

		\item[(ii)] 	For $h=\{h_{t}(\gamma)\}_{t\in \mc N, \gamma\in \Gamma}\in \ell^2(\Gamma\times \mc N),$
		the Zak transform of  $(T_{\mc E^{\Gamma}(\mc A)}^*h)$  at $(\alpha, \Gamma x) \in \widehat \Gamma \times \Gamma \backslash \G$  is 
		$ 		\left[\mc Z (T_{\mc E^{\Gamma}(\mc A)}^*h)\right](\alpha, \Gamma x)= \sum_{t\in \mc N}   \widehat{h}_t (\alpha) \,  \mc Z{\varphi_t}(\alpha,\Gamma x).$ 
 
	Moreover,   $\left\{\left[\mc Z (T_{\mc E^{\Gamma}(\mc A)}^*h)\right](\alpha, \Gamma x)\right\}_{\Gamma x\in \Gamma\backslash \G}=  \mf J_{\mc A}(\alpha) \left\{\widehat{{ h_t}}(\alpha)\right\}_{t\in \mathcal N}\ \mbox{for a.e.} \  \alpha\in\widehat \Gamma.$
		\item[(iii)]  For $f \in L^2(\G)$ and  a.e. $\alpha\in\widehat \Gamma,$
		$	\left\{\mc Z \left(T_{\mc E^{\Gamma}(\mc A)}^*T_{\mc E^{\Gamma}(\mc A')}f \right) (\alpha, \Gamma x)\right\}_{\Gamma x\in\Gamma \backslash \mathscr G} 	=\mf J_{\mc A}(\alpha)\mf J_{\mc A'}(\alpha)^*\left\{\mc Zf(\alpha, \Gamma x)\right\}_{\Gamma x \in\Gamma \backslash \mathscr G}= \tilde G_{\mc A, \mc A'} (\alpha) 	\left\{(\mc Zf)(\alpha,\Gamma x)\right\}_{\Gamma x \in\Gamma \backslash \mathscr G}.$
	\end{itemize}
\end{prop}

\begin{proof} (i) The Fourier transform of   $(T_{\mc E^{\Gamma}(\mc A)}f)_t$ at $\alpha \in \widehat {\Gamma},$  i.e.,    
	$	\widehat{(T_{\mc E^{\Gamma}(\mc A)}f)_t}(\alpha) = [ f,\varphi_t] (\alpha),$  follows by  Lemma \ref{L:matrixelement} and    $(T_{\mc E^{\Gamma}(\mc A)}f)_t=\left\{\langle f, L_{\gamma}\varphi_t\rangle\right\}_{\gamma \in \Gamma}=\left\{(\mc M_{\varphi_t}f)(\gamma)\right\}_{\gamma \in \Gamma}$   for each $t\in \mc N$ and $f \in L^2(\G).$  Further
	  \begin{align*}
  \left\{\widehat{(T_{\mc E^{\Gamma}(\mc A)}f})_{t}
  (\alpha)\right\}_{t\in \mathcal N}
  =\{[ f,\varphi_t] (\alpha)\}_{t\in \mc N}
  &=\left\{\int_{\Gamma \backslash \mathscr G}(\mc Zf)(\alpha,\Gamma x)\overline{\mc Z\varphi_{t}(\alpha, \Gamma x)}\ \dGammax\right\}_{t\in \mc N}\\
  =\left\{\langle \mc Zf(\alpha), \mc Z\varphi_t(\alpha)\rangle\right\}_{t\in \mc N}
  &=\mf J_{\mc A}(\alpha)^*\{(\mc Zf)(\alpha,\Gamma x)\}_{\Gamma x \in\Gamma \backslash \mathscr G} \ \a.e. \ \alpha\in\widehat \Gamma.
  \end{align*}

\noindent (ii) Let $h=\{h_{t}(\gamma)\}_{t\in \mc N, \gamma\in \Gamma}\in \ell^2(\Gamma\times \mc N).$ Employing  
the Zak transform on the synthesis operator   $(T_{\mc E^{\Gamma}(\mc A)}^*h)$  at $(\alpha, \Gamma x) \in \widehat \Gamma \times \Gamma \backslash \G$ and the discrete-Fourier transform on the sequence $\{h_{t}(\gamma)\}_{\gamma\in \Gamma}$ at $\alpha \in \widehat{\Gamma},$ we obtain 
{\small		\begin{align*}
			\mc Z (T_{\mc E^{\Gamma}(\mc A)}^*h)(\alpha, \Gamma x)
	& =\mc Z\left( \sum_{t\in \mc N}\sum_{\gamma \in  \Gamma}h_t(\gamma)L_\gamma\varphi_t\right)(\alpha, \Gamma x)
	=\left( \sum_{t\in \mc N}\sum_{\gamma \in  \Gamma}h_t(\gamma)\mc Z (L_\gamma\varphi_t) \right)(\alpha, \Gamma x)\\
		&= \sum_{t\in \mc N} \left(\sum_{\gamma \in  \Gamma}h_t(\gamma)\ol{\alpha(\gamma)}\right) \mc Z\varphi_t(\alpha, \Gamma x)
		= \sum_{t\in \mc N}\widehat{ h}_t(\alpha) \mc Z{\varphi_t}(\alpha, \Gamma x).
	\end{align*}}
 Then   in terms of pre-Gramian operator for a.e. $\alpha \in \widehat{\Gamma},$ we get 
	$$
	\left\{\mc Z(T_{\mc E^{\Gamma}(\mc A)}^*h)(\alpha, \Gamma x)\right\}_{\Gamma x\in \Gamma \backslash \mathscr G}=\left\{\sum_{t\in \mc N}\widehat{ h}_t(\alpha) \mc Z{\varphi_t}(\alpha, \Gamma x)\right\} _{\Gamma x\in \Gamma\backslash \G}= \mf J_{\mc A}(\alpha)\left\{\widehat{ h}_t (\alpha)\right\}_{t\in \mathcal N}.$$
	\noindent (iii) From the  (i) and (ii) parts associated with $\mc A$ and $\mc A'$ and for a.e. $\alpha \in \widehat \Gamma,$ we get the following by combining both the analysis and synthesis operators  for $f \in L^2 (\G),$
	\begin{align*}
		\left\{\mc Z (T_{\mc E^{\Gamma}(\mc A)}^*T_{\mc E^{\Gamma}(\mc A')}f)(\alpha, \Gamma x)\right\}_{\Gamma x\in\Gamma \backslash \mathscr G}
	 	&=\mf J_{\mc A}(\alpha)\left\{\widehat{{ (T_{\mc E^{\Gamma}(\mc A')}f)_t}}(\alpha)\right\}_{t\in \mathcal N}
		=\mf J_{\mc A}(\alpha)\mf J_{\mc A'}(\alpha)^*\left\{\mc Zf(\alpha, \Gamma x)\right\}_{\Gamma x \in\Gamma \backslash \mathscr G}.
	\end{align*}	
\end{proof}
Now we state main results of this section to characterize subspace orthogonal and dual to a Bessel family having multi-generators.  Theorem \ref{T:dualityMultiGenerators} is a successor of the results of \cite{christensen2005generalized,weber2004orthogonal}  studied for $L^2(\mathbb R^n).$  Theorem \ref{T:Orthogonal}  characterizes orthogonal frames $\mc E^{\Gamma}(\mc A)$ and $\mc E^{\Gamma}(\mc A')$ in terms of pre-Gramian and mixed-dual Gramian operators   for  locally compact groups by action of its  abelian subgroup. The result has so many predecessors   by action of  integer translates in $L^2 (\mathbb R^n)$ and uniform lattices in $L^2 (\mc G)$   \cite{gumber2018orthogonality, gumber2019pairwisebuletin, kim2007pair,  weber2004orthogonal}, where $\mc G$ is an LCA group.	\begin{thm}\label{T:dualityMultiGenerators}  For a   $\sigma$-finite measure space $\mc N$ having counting measure, let  us consider two sequences of functions $\mc A=\{\varphi_t\}_{t\in \mc N}$ and  $\mc A'=\{\psi_t\}_{t\in \mc N}$  in $L^2 (\G)$ be such that the $\Gamma$-TG  systems
		$\mc E^{\Gamma}(\mc A)$ and $\mc E^{\Gamma}(\mc A')$
		are    Bessel. Then the following hold:
		\begin{enumerate}
			\item[(i)]  $\mc E^{\Gamma}(\mc A')$ is an $\mc S^{\Gamma}(\mc A)$-subspace dual  to $\mc E^{\Gamma}(\mc A)$ if and only if  for all $t'\in \mc N,$  we have 
			$$
			(\mc Z\varphi_{t'})(\alpha, \Gamma x)=\sum_{t\in \mc N}  [\varphi_{t'}, {\psi_t}](\alpha)  \,  (\mc Z\varphi_t)(\alpha, \Gamma x)  \ \mbox{for} \ \ {a.e.}\ \alpha \in \widehat \Gamma,\  \Gamma x \in \Gamma \backslash \G.
			$$
		Equivalently	  for   $t'\in \mc N,$ $\{(\mc Z\varphi_{t'})(\alpha, \Gamma x)\}_{\Gamma x \in \Gamma \backslash \G}=\mf J_{\mc A}(\alpha)\left\{[\varphi_{t'}, {\psi_t}](\alpha)\right\}_{t\in \mathcal N}$ for a.e. $\alpha \in \widehat \Gamma.$
				\item [(ii)] $\mc E^{\Gamma}(\mc A')$ is an $\mc S^{\Gamma}(\mc A)$-subspace orthogonal  to $\mc E^{\Gamma}(\mc A)$ if and only if   	for   $ f \in \mc S^\Gamma(\mc A)$ and  $\ g\in L^2(\G),$   
			$$
		\sum_{t\in \mc N}[f,\psi_t](\alpha)[\varphi_t, g](\alpha)=0=	\left\langle \tilde G_{\mc A,\mc A'}(\alpha) (\mc Zf (\alpha)), \mc Zg (\alpha) \right\rangle \ a.e. \  \alpha \in \widehat \Gamma.
			$$
			\end{enumerate}
		Moreover, when the pair $(\G, \Gamma)$ is an abelian pair $(\mc G, \Lambda),$ then (i) and (ii) become (i') and (ii') as follows:
			\begin{enumerate}
			\item[(i')]  $\mc E^{\Lambda}(\mc A')$ is an $\mc S^{\Lambda}(\mc A)$-subspace dual  to $\mc E^{\Lambda}(\mc A)$ if and only if  for all $t'\in \mc N,$\\    $	(\mathscr T \varphi_{t'})(\omega \Lambda^\perp)(\xi)=\sum_{t\in \mc N}  [ \varphi_{t'},   \psi_{t}]_{\mathscr T}(\omega \Lambda^\perp)  \,  \mathscr T \varphi_{t} (\omega \Lambda^\perp)(\xi)$   for a.e. $\omega \Lambda^\perp \in \widehat{\mc G}/\Lambda^\perp$ and $\xi  \in \Lambda^\perp.$
			
			\item [(ii')] $\mc E^{\Lambda}(\mc A')$ is an $\mc S^{\Lambda}(\mc A)$-subspace orthogonal  to $\mc E^{\Lambda}(\mc A)$ if and only if   	for all  $ f \in \mc S^{\Lambda}(\mc A)$  and $\ g\in L^2(\mc G)$ $
			\sum_{t\in \mc N}[ f, \psi_t]_{\mathscr T}(\omega \Lambda^\perp)[\varphi_t, g]_{\mathscr T}(\omega \Lambda^\perp)=0$  for a.e.   $\omega \Lambda^\perp \in \widehat{\mc G}/\Lambda^\perp.$
		\end{enumerate}
	\end{thm}
	\begin{proof} (i) Let $\mc E^{\Gamma}(\mc A')$ be an $\mc S^{\Gamma}(\mc A)$-subspace dual  to $\mc E^{\Gamma}(\mc A).$ Then for $f \in \mc S^{\Gamma}(\mc A),$ we can write  $f=\sum_{t\in \mc N}\sum_{\gamma\in \Gamma} \langle f, L_\gamma \psi_t\rangle L_\gamma\varphi_t.$ By choosing   $f=L_\eta\varphi_{t'}$ for $\eta \in \Gamma$ and $t' \in \mc N,$ and applying the Zak transformation $\mc Z$ on the  both sides, we have 
		$\mc Z (L_\eta\varphi_{t'}) (\alpha, \Gamma x)= \mc Z\left(\sum_{t\in \mc N} \sum_{\gamma\in \Gamma} \langle L_\eta\varphi_{t'}, L_\gamma \psi_t\rangle L_\gamma\varphi_t \right) (\alpha, \Gamma x)$ for a.e.  $\alpha \in \widehat \Gamma$ and $\Gamma x \in \Gamma \backslash \G.$  Thus we get the result   by noting  
     $\mc Z (L_\eta\varphi_{t'}) (\alpha, \Gamma x) = \alpha(\eta^{-1}) \mc Z \varphi_{t'} (\alpha, \Gamma x),$  and   Proposition  \ref{P:zakpre-gramian} (iii),
     \begin{align*}
     	\mc Z(T_{\mc E^\Gamma(\mc A)}^*T_{{\mc E^\Gamma}(A')}L_\eta \varphi_{t'})(\alpha, \Gamma x)
     	&=\mf J_{\mc A}(\alpha)\mf J_{\mc A'}(\alpha)^*\left\{\mc ZL_\eta\varphi_{t'}(\alpha, \Gamma x)\right\}_{\Gamma x \in\Gamma \backslash \mathscr G}\\
     	&=\sum_{t\in \mc N}  [\mc Z (L_\eta\varphi_{t'}),  \psi_t] (\alpha)\mc Z \varphi_{t} (\alpha, \Gamma x)\\
		 &	= \alpha(\eta^{-1}) \sum_{t\in \mc N}[\varphi_{t'},  \psi_t](\alpha) \mc Z \varphi_t (\alpha, \Gamma x).
	\end{align*}
Conversely, assume     $(\mc Z\varphi_{t'})(\alpha, \Gamma x)=\sum_{t\in \mc N}  [\varphi_{t'}, {\psi_t}](\alpha)  \,  (\mc Z\varphi_t)(\alpha, \Gamma x)$    a.e. $\alpha \in \widehat \Gamma,$   $\Gamma x \in \Gamma \backslash \G$ and     $t'\in \mc N.$  Then we have  $L_\eta\varphi_{t'} = \sum_{t\in \mc N}\sum_{\gamma\in \Gamma}\langle L_\eta\varphi_{t'}, L_\gamma \psi_t\rangle L_\gamma\varphi_t$ in view of the above calculations for $\eta \in \Gamma$ and $t' \in \mc N.$ Therefore for $f \in \mbox{span} \ \mc E^{\Gamma}(\mc A),$  we can write  
	$f =\sum_{t\in \mc N} \sum_{\gamma\in \Gamma} \langle f, L_\gamma \psi_t\rangle L_\gamma\varphi_t$ which is also valid for all $f \in \mc S^\Gamma (\mc A)$ since  the function $f\mapsto \sum_{t\in \mc N} \sum_{\gamma\in \Gamma} \langle f, L_\gamma \psi_t\rangle L_\gamma\varphi_t$ from $\mc S^\Gamma (\mc A)$ to  $L^2 (\G)$ is continuous due to the Bessel systems   $\mc E^\Gamma(\mc A)$  and $\mc E^\Gamma(\mc A').$   Thus the result holds.

	For the equivalent part, $\left\{[\varphi_{t'}, {\psi_t}](\alpha)\right\}_{t\in \mathcal N}$ is a members of $\ell^2 (\mc N)$ for a.e. $\alpha \in \widehat \Gamma,$ follows from
	\begin{align*}
\sum_{t \in \mc N} 	\int_{\widehat \Gamma}  \left| [\varphi_{t'}, {\psi_t}](\alpha)\right|^2\  \dalpha &=   \sum_{t \in \mc N} 	\int_{\widehat \Gamma} \left|	\widehat{(\mc M_{\psi_{t}} \varphi_{t'})}(\alpha) \right|^2 \ \dalpha=\sum_{t \in \mc N} \sum_{\gamma \in \Gamma}   \left|(\mc M_{\psi_{t}} \varphi_{t'})(\gamma) \right|^2\\
&= \sum_{t \in \mc N} \sum_{\gamma \in \Gamma}   \left|\left \langle \varphi_{t'}, L_\gamma \psi_t \right\rangle  \right|^2 \leq B \|\varphi_{t'}\|^2
	\end{align*}
for some $B>0,$ since $\mc E^{\Gamma}(\mc A')$ is Bessel.  Using the definition of $\mf J_{\mc A}(\alpha)$ we get the result.

\noindent(ii)  	Suppose  $\mc E^{\Gamma}(\mc A')$ is an $\mc S^{\Gamma}(\mc A)$-subspace orthogonal to $\mc E^{\Gamma}(\mc A),$  then for $f\in \mc S^\Gamma(\mc A)$ and $g\in  L^2(\G),$  we have 
$\sum_{ t\in \mc N}\sum_{\gamma \in \Gamma} \langle f, L_\gamma \psi  \rangle  \langle L_\gamma \varphi, g  \rangle  = 0 $
from (\ref{orthogonal}), and hence  we can get the following easily  
\begin{align}\label{eq:ortho-discre}
\sum_{ t\in \mc N} \int_{\widehat {\Gamma}}[f, \psi_t](\alpha) [\varphi_t,g](\alpha)  \ d{\mu_{\widehat \Gamma}}(\alpha)=0, 
\end{align}
by considering countable functions in     Lemma \ref{L:BesselCharecterization} .
Therefore for a.e. $\alpha \in \widehat{\Gamma},$ we need to  prove \\ $\sum_{ t\in \mc N} [f, \psi_t](\alpha)  [\varphi_t,g](\alpha)  =0$  which is same as   $\sum_{ t\in \mc N}\langle \mc Zf(\alpha), \mc Z\psi_t(\alpha)\rangle  \langle  \mc Z\varphi_t(\alpha), \mc Z g(\alpha)\rangle =0$ for a.e. $\alpha.$ 
For this, let $(e_i)_{i\in \mathbb Z}$ be an  orthonormal basis for $L^2(\Gamma\backslash \G)$ and $P(\alpha)$ be an orthogonal projection of $L^2(\Gamma\backslash \G)$ on $\ol{\Span}\{\{\mc Z\varphi_t(\alpha, \Gamma x)\}_{\Gamma x \in \Gamma \backslash \G}:t\in \mc N\}$ for a.e. $\alpha \in \widehat{\Gamma}.$ Assume on the contrary,  there exists $i_0\in \mathbb Z$ such that 
$$h(\alpha)=\sum_{ t\in \mc N} \langle P(\alpha) e_{i_0}, \mc Z\varphi_t(\alpha)\rangle   \langle\mc Z\psi_t(\alpha) ,  \mc Zg (\alpha)\rangle \neq 0$$ on a measurable set $E\ss \widehat{\Gamma}$ with $\mu_{\widehat \Gamma}(E)>0.$ Then one of the four sets must have positive measure:
\begin{align*}
E_1=\{\alpha\in E: \mbox{Re}~h(\alpha)>0\} & \qquad \qquad  E_3=\{\alpha\in E: \mbox{Im}~ h(\alpha)>0\}\\
E_2=\{\alpha\in E: \mbox{Re}~ h(\alpha)<0\} & \qquad \qquad E_4=\{\alpha\in E: \mbox{Im}~ h(\alpha)<0\}.
\end{align*}
Suppose $\mu_{\widehat {\Gamma}}(E_1)>0,$  and choosing  $f\in \mc S^\Gamma(\mc A)$ such that   for    all $\Gamma x \in \Gamma \backslash \G,$  $\mc Zf(\alpha, \Gamma x)= P(\alpha)e_{i_{0}}$   for a.e. $\alpha \in E_1$  and zero for  other $\alpha$'s. Then   the estimate 
$$
Re\left\{\sum_{ t\in \mc N}  \left(\int_{\widehat {\Gamma}} [f, \psi_t](\alpha) [\varphi_t,g](\alpha) \ d{\mu_{\widehat \Gamma}}(\alpha)\right)\right\} >0
$$
due to $\mu_{\widehat {\Gamma}}(E_1)>0,$ which  contradicts the fact that  the integration is zero by (\ref{eq:ortho-discre}). Similarly, we can discuss for the other sets $E_2,$ $E_3,$  $E_4,$ and will arrive on the same conclusion.  The converse part follows immediately by Lemma \ref{L:BesselCharecterization}.   The remaining part follows easily from the definition of $\tilde G_{\mc A, \mc A'}(\alpha)$ for a.e. $\alpha \in \widehat \Gamma.$
 	\end{proof}
	 The following immediate consequence  can be observed easily by the   Theorem \ref{T:dualityMultiGenerators}.
	 \begin{cor}\label{Cor: Multi-dual} For  $t, {t'}\in \mc N,$ let  $[ \varphi_{t'}, \psi_{t}](\alpha)=\delta_{t,t'}$ for a.e. $\alpha \in \widehat \Gamma$ along with the assumptions of Theorem \ref{T:dualityMultiGenerators}. Then  $\mc E^{\Gamma}(\mc A')$ is an $\mc S^{\Gamma}(\mc A)$-subspace dual  to $\mc E^{\Gamma}(\mc A).$  
	 \end{cor}
 The following result describes few more  properties of orthogonal frames using mixed dual-Gramian operator, which is a successor of  \cite[Theorem 3.2]{gumber2019pairwisebuletin} and \cite[Theorem 2.2]{kim2007pair}.
 \begin{thm}\label{T:Orthogonal} Let $\mc A=\{\varphi_t\}_{t\in \mc N}$ and $\mc A'=\{\psi_t\}_{t\in \mc N}$  be two sequences of  functions in $L^2 (\mathscr G)$ such that   $\mc  E^{\Gamma}(\mc A)$ and $\mc  E^{\Gamma}(\mc A')$ are $\mc S^\Gamma(\mc A)$ and $\mc S^\Gamma(\mc A')$-subspace frames, respectively. If   $\mc S^\Gamma(\mc A)=\mc S^\Gamma(\mc A'),$   the following are equivalent:
 	\begin{enumerate}
 		\item[(i)]   $\mc  E^{\Gamma}(\mc A)$ and $\mc  E^{\Gamma}(\mc A')$ are  orthogonal  pair. 
 		\item[(ii)]  $\mf J_{\mc A}(\alpha)\mf J_{\mc A'}(\alpha)^*\mf J_{\mc A'}(\alpha)=0$ for a.e. $\alpha \in \widehat \Gamma.$
 		\item[(iii)] $G_{\mc A}(\alpha) G_{\mc A'}(\alpha)=0$ for a.e. $\alpha\in \widehat \Gamma.$ 
 	\end{enumerate} 
 	Additionally, when $\mc S^\Gamma(\mc A)=\mc S^\Gamma(\mc A') =L^2(\mathscr G),$ then $\mc  E^{\Gamma}(\mc A)$ and $\mc  E^{\Gamma}(\mc A')$ are  orthogonal  pair if and only if $ \tilde G_{\mc A,\mc A'}(\alpha)=0$ for a.e. $\alpha \in \widehat \Gamma.$
 \end{thm} 
 
We proceed by decomposing any  $\Gamma$-TI space as an   orthogonal direct sum of  $\mc S^{\Gamma}(\varphi_i)$'s up to countable, where $\mc E^{\Gamma}(\varphi_i)$ is a   frame for $\mc S^{\Gamma}(\varphi_i)$ with bounds $A=1$ and $B=1$ for each $i.$  Our procedure is motivated by DeBoor, DeVore, and Ron \cite{deboor1994structure} and  \cite[Theorem 3.3]{bownik2000structure}. Their analysis relied on the Fourier transform, whereas ours makes use of the Zak transform.
 \begin{prop}\label{P:decomposition} For   $\varphi \in L^2(\mathscr G),$  a function	$f\in \mc S^{\Gamma}(\varphi)$ if and only if for a.e. $\alpha \in \widehat \Gamma,$  
 	$\mc Zf(\alpha, \Gamma x)=\mathfrak m(\alpha)\mc Z\varphi(\alpha, \Gamma x)$
 	for all $\Gamma x \in \Gamma \backslash \G,$ where  $\mathfrak m$ is a member of the weighted space $ L^2 \left(\widehat \Gamma, [\varphi, \varphi] \right).$ 

 	Moreover, if $V$ is an $\Gamma$-TI subspace of  $L^2(\mathscr G)$ then  
 	there are  at most countably many   $\varphi_n's$ in $V$ such that 
 	$f\in V$ can be decomposed   as follows for a.e. $\alpha \in \widehat \Gamma$: 
 	\begin{equation}\label{EQ:Zf=}
 	\mc Zf(\alpha, \Gamma x)=\sum_{n\in \mathbb N}\mathfrak m_n(\alpha)\mc Z \varphi_n(\alpha, \Gamma x)  \  \mbox{for all}\  \Gamma x \in \Gamma \backslash \G,  
 	\end{equation}
 	where  $\mf m_n\in L^2(\widehat \Gamma\cap \Omega_{\varphi_n})$ and    $	\Omega_{\varphi_n}=\{\alpha \in \widehat \Gamma: [\varphi_n,\varphi_n]\neq 0\}  .$ 
 \end{prop}
 \begin{proof} For   $f\in \mbox{span}\ \mc E^{\Gamma}(\varphi),$ a representation of  $f$ is   of the form $f=\sum_{\gamma \in \Gamma} c_\gamma L_{\gamma}\varphi ,$  where only  finitely many $c_{\gamma}$'s are  non-zero, and hence by applying the Zak transform on both the sides, we obtain,
 	\begin{align*} 
 	(\mc Zf)(\alpha, \Gamma x) 
 	=\sum_{\gamma\in\Gamma} c_\gamma\mc Z\varphi(\alpha, \Gamma x)\overline{\alpha(\gamma)}
 	=\mc Z\varphi(\alpha, \Gamma x)\sum_{\gamma\in\Gamma}c_\gamma\overline{\alpha(\gamma)}
 	=\mathfrak m(\alpha)\mc Z\varphi(\alpha, \Gamma x) 
 	\end{align*}
 	for a.e. $\alpha \in \widehat \Gamma,$ and $\Gamma x \in \Gamma \backslash \G,$   	where $\mathfrak m:\widehat \Gamma\mapsto \mathbb C,$ $\mathfrak m(\alpha):=\sum_{\gamma\in\Gamma}
 	c_\gamma\overline{\alpha(\gamma)}.$
 	Conversely,   we can recover  $f\in \mbox{span}\ \mc E^{\Gamma}(\varphi)$  by above relation. It only  remains to generalize it for $f\in \mc S^{\Gamma}(\varphi).$ For this 
 	define an operator  $ \msc  U : \mbox{span}  \ \mc E^{\Gamma}(\varphi) \ra \mc P$ by $\msc Uf=\mf m,$ where $\mc P$ is the collection of  all trigonometric polynomials, which  is an isometry and onto, follows by 
 	{\small 	\begin{align}\label{eq:normf}
 		\|f\|^2
 		=\int_{\widehat \Gamma}\int_{\Gamma \backslash \G}| \mf m(\alpha)\mc Z\varphi(\alpha,\Gamma x)|^2\ \dGammax
 		\dalpha  
 		=\int_{\widehat \Gamma}| \mf m(\alpha)|^2 [ \varphi, \varphi](\alpha)
 	\	\dalpha  
 		=\|\mf m\|_{{L^2(\widehat \Gamma ,[ \varphi, \varphi]})}^2.
 		\end{align}}
 	Therefore, there exists a unique isometry $\tilde{\msc U}: \mc S^{\Gamma}(\varphi) \ra  \ol{\mc P}=L^2(\widehat \Gamma,[ \varphi, \varphi]).$
 	The moreover  parts follow by observing orthogonal projections $P_n$'s on $\mc S^{\Gamma}(\varphi_n)$ and the following calculations   for every $f\in V$:
 	$$ 	\mc Zf(\alpha, \Gamma x)=\sum_{n\in \mathbb N}\mc Z(P_nf)(\alpha, \Gamma x)=\sum_{n\in \mathbb N}\mathfrak m_n(\alpha)\mc Z \varphi_n(\alpha, \Gamma x)
 	, \ \mf m_n \in L^2(\widehat \Gamma\cap \Omega_{\varphi_n}),$$ 
 	a.e.   \ $\alpha \in \widehat \Gamma$ and $\Gamma x \in \Gamma \backslash \G.$ Thus the result follows. 
 	 \end{proof}
 \begin{proof} [{Proof of  Theorem \ref{T:Orthogonal}}] The $\Gamma$-TG systems 	$\mc  E^{\Gamma}(\mc A)$ and $\mc  E^{\Gamma}(\mc A')$  are   orthogonal   if and only if $T_{\mc E^{\Gamma}(\mc A)}^* T_{\mc E^{\Gamma}(\mc A')}f=0$  for $f\in \mc S^\Gamma(\mc A).$  Equivalently, 
	\begin{align*}
	0=&	 \|\mc Z (T_{\mc E^{\Gamma}(\mc A)}^* T_{\mc E^{\Gamma}(\mc A')})f\|^2 
	= \int_{\widehat \Gamma} [(T_{\mc E^{\Gamma}(\mc A)}^* T_{\mc E^{\Gamma}(\mc A')})f,  (T_{\mc E^{\Gamma}(\mc A)}^* T_{\mc E^{\Gamma}(\mc A')})f] (\alpha)\ \dalpha\\
	&=\int_{\Gamma \backslash \G}\int_{\widehat \Gamma}|\mc ZT_{\mc E^{\Gamma}(\mc A)}^* T_{\mc E^{\Gamma}(\mc A')}f(\alpha, \Gamma x)|^2\ \dalpha\ \dGammax.
	\end{align*}
 Further, it is equivalent to    $$\{\mc Z (T_{\mc E^{\Gamma}(\mc A)}^* T_{\mc E^{\Gamma}(\mc A')})f(\alpha,\Gamma x)\}_{\Gamma x\in \Gamma\backslash \G}=0  \ \mbox{for a.e.} \  \alpha \in \widehat \Gamma.
	$$
	Since $
	\left\{\mc Z (T_{\mc E^{\Gamma}(\mc A)}^*T_{\mc E^{\Gamma}(\mc A')} f)(\alpha, \Gamma x)\right\}_{\Gamma x\in\Gamma \backslash \mathscr G}
	=\mf J_{\mc A}(\alpha)\mf J_{\mc A'}^*(\alpha)\left\{\mc Zf(\alpha, \Gamma x)\right\}_{\Gamma x \in\Gamma \backslash \mathscr G}
	$
	by Proposition \ref{P:zakpre-gramian},  and also from Proposition \ref{P:decomposition}, $f\in \mc S^\Gamma(\mc A')$    if and only if  
	$\{\mc Zf(\alpha,\Gamma x)\}_{\Gamma x \in \Gamma \backslash \G}=\left\{\sum_{t\in \mc N} \mf m_t(\alpha)\mc Z\psi_t(\alpha, \Gamma x)\right\}_{\Gamma x\in \Gamma\backslash \G}=\mf J_{\mc A'}(\alpha)\left\{\mf m_t(\alpha)\right\}_{t\in \mc N}$  for a.e. $\alpha \in \widehat \Gamma.$
	Therefore we get   
	$$\left\{\mc ZT_{\mc E^{\Gamma}(\mc A)}^*T_{\mc E^{\Gamma}(\mc A')} f(\alpha, \Gamma x)\right\}_{\Gamma x\in\Gamma \backslash \mathscr G}
	=\mf J_{\mc A}(\alpha)\mf J_{\mc A'}(\alpha)^* \mf J_{\mc A'}(\alpha)\left\{\mf m_t(\alpha)\right\}_{t\in \mc N} \ \mbox{a.e.} \ \alpha \in \widehat \Gamma.$$
	Thus (i) is equivalent to (ii) follows by observing that $f$ is an arbitrary member of $\mc S^\Gamma(\mc A).$  	The equivalency between (ii) and (iii) follows immediately by just observing frame property of  $\mc E^\Gamma(\mc A).$

	When  $\mc S^\Gamma(\mc A)=L^2(\G),$ $\mf J_{\mc A'}(\alpha)^*$ has bounded inverse on the range of $\mf J_{\mc A'}(\alpha),$ and hence the result follows.
\end{proof}

Next, we observe that a new orthogonal pair can be constructed  from the given orthogonal  pair by involving  $\Gamma$-periodic functions. A function $f:\G \ra 
\mathbb C$ is said to be \textit{$\Gamma$-periodic} if $f(x+\gamma)=f(x)$ for all $\gamma \in \Gamma, x\in \G.$ 
\begin{prop} Under the assumptions of Theorem \ref{T:dualityMultiGenerators}, let $\mc  E^{\Gamma}(\mc A')$ be an $\mc S^\Gamma(\mc A)$-subspace orthogonal to $\mc  E^{\Gamma}(\mc A).$ If $h$ is  an $\Gamma$-periodic function on $\G,$    $\mc  E^{\Gamma}(h\mc A')$ is also  an $\mc S^\Gamma(\mc A)$-subspace orthogonal to $\mc  E^{\Gamma}(\mc A),$ where   $h\mc A'=\{h\psi :   \psi \in \mc A',    (h\psi)(x)=h(x)\psi(x), x\in \G \}.$
\end{prop}
\begin{proof} The result follows by observing   
	$\sum_{t\in \mc N}\sum_{\gamma \in\Gamma}\langle f, \psi_t(x-\gamma)\rangle \varphi_t(x-\gamma)=0,$ and  $$ \sum_{t\in \mc N}\sum_{\gamma\in \Gamma}\langle f, (h\psi_t)(x-\gamma)\rangle \varphi_t(x-\gamma)= \overline{h(x)}\sum_{t\in \mc N}\sum_{\gamma\in \Gamma}\langle f, \psi_t(x-\gamma)\rangle \varphi_t(x-\gamma)$$
	for $x \in \G$ and  $f\in L^2(\G).$
\end{proof}
  \subsection{Application to singly generated systems} \label{ss:disc-singlegen}
 For  a function $\varphi\in L^2 (\G),$ we recall the $\Gamma$-TG system   $\mc E^{\Gamma} (\varphi)$ and its associated  $\Gamma$-TI space $\mc S^{\Gamma}(\varphi)$ from (\ref{TIsystem}). The following consequences of Theorem   \ref{T:dualityMultiGenerators}, state about an $\mc S^{\Gamma}(\varphi)$-subspace duals/ orthogonal to the system $\mc E^{\Gamma} (\varphi).$ 

The following Corollary is a generalization to the   locally compact  group setup of \cite[Theorem 4.1]{christensen2004oblique} 
\cite[Corrollary 4.6]{heil2009duals}, \cite[Proposition 1]{christensen2005generalized}  for $L^2(\mathbb R^n)$ by the action of integer   translations.
 	\begin{cor}\label{C:DualMain}
		Let  $\varphi$ and $\psi$ be two functions in $L^2(\mathscr G)$   such that the   corresponding $\Gamma$-TG   systems $\mc E^{\Gamma}(\varphi)$ and $\mc E^{\Gamma}(\psi)$ are  Bessel. Assume a  measurable set  $ \Omega_{\varphi}$ defined by 	$ \Omega_{\varphi}:=\left\{\alpha \in \widehat \Gamma: [ \varphi, \varphi](\alpha)  \neq 0\right\}.$  Then the following are true:
		\begin{enumerate}
			\item [(i)]   $\mc E^{\Gamma}(\psi)$ is an $\mc S^{\Gamma}(\varphi)$-subspace dual to $\mc E^{\Gamma}(\varphi)$ if and only if   
			$
			[ \varphi, \psi](\alpha)
			=1  \  \a.e.  \ \alpha \in   \Omega_{\varphi}. 
			$

			\item [(ii)]  $\mc E^{\Gamma}(\psi)$ is an $\mc S^{\Gamma}(\varphi)$-subspace orthogonal to $\mc E^{\Gamma}(\varphi)$ if and only if    
			$
			[ \varphi,  \psi](\alpha)
			=0 \ a.e.  \ \alpha \in   \Omega_{\varphi}. 
			$
			In case of either  one holds, we have $[\varphi,  \psi](\alpha)=0$ for a.e. $\alpha \in \widehat \Gamma,$ which implies  $\mc E^{\Gamma}(\varphi)$  is also an   $\mc S^{\Gamma}(\psi)$-subspace orthogonal to $\mc E^{\Gamma}(\psi).$ 
		\end{enumerate}
%
Moreover, the same can be deduced in terms of the fiberization  $\mathscr T$ for an abelian pair $(\mc G, \Lambda).$
	\end{cor}
\begin{proof}
The result  follows easily  by choosing $\varphi_t=\varphi,$ and $\psi_t=\psi$ for every $t$ in  Theorem \ref{T:dualityMultiGenerators} and $f=g=\varphi_t.$

Next assume that   $[ \varphi,  \psi](\alpha)=0$ for a.e. $\alpha \in \Omega_{\varphi}.$ Then  first note that  $[\varphi,\varphi](\alpha)=0$  on  a.e. $\widehat \Gamma \backslash \Omega_{\varphi},$ and hence using the  Cauchy-Schwarz inequality in the following estimate   for a.e. $\alpha \in \widehat \Gamma,$ 
\begin{align} \label{eq:cauchyschwartzbracket}
|[\varphi,\psi](\alpha)| 	
&\leq	\int_{ \Gamma \backslash \mathscr G}\left|\mc Z\varphi(\alpha,\Gamma x)\ol{\mc Z \psi(\alpha,\Gamma x)}\right|\ \dGammax\nonumber \\
&	\leq \left(\int_{ \Gamma \backslash \mathscr G}	\left|\mc Z\varphi(\alpha,\Gamma x)\right|^2\dGammax \right)^{1/2}\left( \int_{ \Gamma \backslash \mathscr G}	\left|\mc Z\psi(\alpha,\Gamma x)\right|^2\ \dGammax\right)^{1/2}\nonumber\\
&=([\varphi, \varphi](\alpha))^{1/2}([\psi,\psi](\alpha))^{1/2},  
\end{align}
we get  $[\varphi,\psi](\alpha)=0$ on a.e. $\widehat \Gamma \backslash \Omega_{\varphi}.$ Thus, we have $[\varphi,\psi](\alpha)=0$ for  a.e. $\alpha \in \widehat \Gamma.$ 

The moreover part  follows from  the same argument as above by replacing the Zak transform $\mc Z$ for the pair $(\G, \Gamma)$ with the fiberization $\mathscr T$ for the pair $(\mc G, \Lambda)$ 
\end{proof}
The part (ii) in Corollary \ref{C:DualMain}  motivates to elaborate more regarding the symmetry of $\varphi$ and $\psi.$ For part (i), having same nature a  counter example is provided in Example \ref{ex:discrete-dual2}.
 We provide various   necessary and  sufficient conditions on the $\Gamma$-TG systems   to become   orthogonal pairs. It is a generalization of the result \cite[Corrollary 2.7]{weber2004orthogonal} for locally compact group.
 \begin{thm}\label{T:orthogonalsupplementary} 	Let  $\varphi$ and $\psi$ be two functions in $L^2(\mathscr G)$   such that the     $\Gamma$-TG   systems $\mc E^{\Gamma}(\varphi)$ and $\mc E^{\Gamma}(\psi)$ are  Bessel.  Then the following are true: 
 	\begin{enumerate}
 		\item [(i)]  Assume   $[\varphi,  \varphi](\alpha)[\psi,  \psi](\alpha)=0$ for a.e. $\alpha \in \widehat \Gamma.$  Then 	$[ \varphi, \psi](\alpha)=0$ a.e. $\alpha,$  and hence $\mc E^{\Gamma}(\psi)$ is a  $\mc S^{\Gamma}(\varphi)$-subspace orthogonal to $\mc E^{\Gamma}(\varphi).$
 		\item [(ii)]  If   $\mc E^{\Gamma}(\psi)$ is an $\mc S^{\Gamma}(\varphi)$-subspace orthogonal to $\mc E^{\Gamma}(\varphi),$  then  $\mc E^{\Gamma}(\tilde {\psi})$ is also so for all   $\tilde {\psi} \in \mc S^\Gamma(\psi).$ 
 		\item[(iii)]	 If   $\mc S^{\Gamma}(\psi)=\mc S^{\Gamma}(\varphi)$  and $\mc E^{\Gamma}(\psi)$ is an  $\mc S^{\Gamma}(\varphi)$-subspace orthogonal to $\mc E^{\Gamma}(\varphi),$  then $\mc E^{\Gamma}(\varphi)$ is also an $\mc S^{\Gamma}(\varphi)$-subspace orthogonal to $\mc E^{\Gamma}(\psi),$ and hence $\mc E^{\Gamma}(\psi)$ and $\mc E^{\Gamma}(\varphi)$ are orthogonal pair. 
 		\item[(iv)] Assume that   $\mc S^{\Gamma}(\psi)=\mc S^{\Gamma}(\varphi).$ Then      $\mc E^{\Gamma}(\psi)$ and $\mc E^{\Gamma}(\varphi)$ are orthogonal pair  if and only if   for a.e. $\alpha \in \widehat \Gamma,$ 
 		$$[\varphi, \varphi](\alpha)[ \psi, \psi](\alpha)=0.$$  
 		\item [(v)]  If the functions $\varphi$ and $\psi$ satisfy $(\supp  \mc Z \varphi) \cap (\supp  \mc Z \psi)=0$ a.e., then $\mc E^{\Gamma}(\psi)$ is an  $\mc S^{\Gamma}(\varphi)$-subspace orthogonal to $\mc E^{\Gamma}(\varphi),$ where  $\supp  \mc Z \varphi$ denotes the support of $\mc Z \varphi$ by considering the map $\mc Z \varphi: \widehat \Gamma \rightarrow L^2 (\Gamma \backslash \G).$ 
 	\end{enumerate}
 \end{thm}

 \begin{proof} (i) The expression $ [ \varphi,  \psi](\alpha)=0$  follows by observing the estimate from (\ref{eq:cauchyschwartzbracket})  for a.e. $\alpha \in \widehat \Gamma,$ 
 	\begin{align*} 
 	|[\varphi,\psi](\alpha)| 	
 	\leq	\int_{ \Gamma \backslash \mathscr G}\left|\mc Z\varphi(\alpha,\Gamma x)\ol{\mc Z \psi(\alpha,\Gamma x)}\right|\ \dGammax 
 	\leq ([\varphi, \varphi](\alpha))^{1/2}([\psi,\psi](\alpha))^{1/2}
 	\end{align*}
 	using Cauchy-Schwarz inequality. From  Corollary \ref{C:DualMain},   $\mc E^{\Gamma}(\psi)$ is  an  $\mc S^{\Gamma}(\varphi)$-subspace orthogonal to $\mc E^{\Gamma}(\varphi).$

 	\noindent(ii) For  $\tilde \psi\in  S^\Gamma(\psi),$ we  can write $\mc Z\tilde \psi(\alpha,\Gamma x)=\mf m(\alpha)\mc Z\psi(\alpha,
 	\Gamma x)$  for a.e.$ \ {\alpha \in \widehat \Gamma} ,$  and $\Gamma x \in \Gamma \backslash \G$  due to Proposition \ref{P:decomposition}, where $\mf m$ is a member of the  weighted space $L^2 (\widehat \Gamma, [  \psi,  \psi]).$ Then,  we get 
 	\ $[\varphi,\tilde\psi](\alpha)=\ol{\mf m(\alpha)}[\varphi, \psi](\alpha),$ and hence the result follows  since the system     $\mc E^{\Gamma}(\psi)$ is  an  $\mc S^{\Gamma}(\varphi)$-subspace orthogonal to $\mc E^{\Gamma}(\varphi),$ equivalently, $[\varphi,\psi](\alpha)=0$ a.e. $\alpha \in \widehat \Gamma$ by Corollary \ref{C:DualMain}.
 	
 	\noindent(iii) This follows easily by  Corollary \ref{C:DualMain}.
 	
 	\noindent (iv)  It is enough to show   the orthogonality of  the Bessel  pairs $\mc E^{\Gamma}(\varphi)$  and $\mc E^{\Gamma}(\psi)$  implies the expression $[ \varphi,  \varphi](\alpha)[ \psi,  \psi](\alpha)$ become zero for a.e. $\alpha$ in view of Corollary \ref{C:DualMain} and part (i).   For this, we proceed similar to the part (i) of Corollary \ref{C:DualMain} by considering $
 	\sum_{ \gamma \in \Gamma}\langle f, L_\gamma\psi\rangle L_\gamma\varphi =0$ for all $f \in  \mc S^\Gamma(\varphi)=\mc S^ \Gamma(\psi).$ Then we get $[ \psi, \psi] (\alpha) \mc Z \varphi (\alpha, \Gamma x)=0$ by choosing $f=\psi$ and hence, we obtain either   $[ \psi, \psi] (\alpha) =0$ or $[ \varphi,  \varphi] (\alpha) =0$ for a.e. $\alpha \in \widehat \Gamma.$ This proves the result.
 	
 	\noindent (v) Observe that the set $\{\alpha \in \widehat \Gamma: \{\mc Z \varphi (\alpha, \Gamma x)\}_{\Gamma x \in \Gamma\backslash \G} \neq 0\}$ is same as the set $\{\alpha \in \widehat \Gamma: \|\{\mc Z \varphi (\alpha, \Gamma x)\}_{\Gamma x \in \Gamma\backslash \G} \|\neq 0\}$ which is further equal to $\{\alpha \in \widehat \Gamma: [ \varphi,  \varphi] (\alpha) \neq 0\}.$ Therefore, the support of  $\mc Z \varphi$  is same as the support of $[ \varphi,  \varphi].$  Hence    $(\mbox{supp} \ \mc Z \varphi) \cap (\mbox{supp} \ \mc Z \psi)=0$ a.e.  implies $(\mbox{supp} \ [ \varphi,  \varphi]) \cap (\mbox{supp} \ [ \psi,  \psi])=0$ a.e.  Thus, we get $[ \varphi,  \varphi](\alpha)[ \psi,  \psi](\alpha)=0$ for a.e. $\alpha \in \widehat \Gamma.$  Now from part (i) we get result. 		 
 \end{proof}
Now we provide some examples to illustrate our results.
\begin{exa} \label{ex:discrete-dual2}
	For a second countable   LCA group $\mc G$ having uniform lattice   $\Lambda,$ we can write $\widehat{\mc G}=\Omega \oplus  \Lambda^{\perp}$ due to   the Pontryagin duality theorem, where $\Omega$ is a \textit{fundamental domain}.    Then, the system  $\{\Omega+\lambda:\lambda \in \Lambda^{\perp}\}$
	is a measurable partition of $\widehat {\mc G}.$ Note that $\Omega$ is a Borel section of $\widehat {\mc G}/\Lambda^\perp.$ 
	\begin{enumerate}
		\item [(i)] \textbf{Subspace dual}:  	Let  $ \eta_1, \eta_2\in L^2(\mc G)$ be such that   $\widehat \eta_1=
		\chi_{A \Omega}$ and $\widehat \eta_2=
		\chi_{\Omega},$ where  $A$ is   an automorphism   on  $\widehat {\mc G}$  such that $A \Omega \subsetneq \Omega.$  Then for $i=1, 2,$      $\mc E^{\Lambda}(\eta_i)$ is   an $\mc S^{\Lambda}(\eta_i)$-subspace  frame, and  $\mc E^{\Lambda}(\eta_2)$ is an $\mc S^\Lambda(\eta_1)$-subspace dual to  $\mc E^{\Lambda}(\eta_1)$ since  for a.e. $\xi \in \Omega,$ $\mathscr T \eta_i (\xi)=\{\widehat{\eta_i}(\xi+\lambda)\}_{\lambda \in \Lambda^\perp} ,$ $ [\eta_1,  \eta_1]_{\mathscr T}(\xi)= 
		\sum_{\lambda \in \Lambda ^{\perp}}|\widehat \eta_1(\xi + \lambda)|^2= \chi_{A\Omega} (\xi), $ $[ \eta_2,  \eta_2]_{\mathscr T}(\xi)= \sum_{k \in \Lambda ^{\perp}}|\widehat \eta_2 (\xi + \lambda)|^2= \chi_{\Omega} (\xi), $  and  $[ \eta_2,  \eta_1]_{\mathscr T}(\xi)=\sum_{\lambda \in \Lambda ^{\perp}} \widehat \eta_2(\xi + \lambda) \overline{\widehat{\eta}_1 (\xi + \lambda)}=\chi_{A\Omega} (\xi)$ by applying Corollary \ref{C:DualMain}. Note that $\mc E^{\Lambda}(\eta_1)$ is not  an $\mc S^\Lambda(\eta_2)$-subspace dual to $\mc E^{\Lambda}(\eta_2).$ 
		
		\item  [(ii)] \textbf{Orthogonal Bessel pair}:      
		Let   $ \eta_1, \eta_2\in L^2(\mc G)$ be such that   $\widehat \eta_1=
		\chi_{\Omega_1}$ and $\widehat \eta_2=
		\chi_{\Omega_2},$ where $\mu_{\widehat{\mc G}} (\Omega_1 \cap \Omega_2)=0,$ and for each $i=1, 2,$ the system  $\{\Omega_i+\lambda:\lambda \in \Lambda^{\perp}\}$
		is a measurable partition of $\widehat {\mc G}.$  Then for each $i=1, 2,$      $\mc E^{\Lambda}(\eta_i)$ is   an $\mc S^{\Lambda}(\eta_i)$-subspace frame  since   $ [\eta_i,  \eta_i]_{\mathscr T}(\xi)= 1$ for a.e. $\xi \in \Omega_i.$ Further note that  $\mc E^{\Lambda}(\eta_2)$  is  an $\mc S^{\Lambda}(\eta_1)$-subspace orthogonal to $\mc E^{\Lambda}(\eta_1)$  since for a.e. $\xi \in \Omega_1,$  
		$[ \eta_2,  \eta_1]_{\mathscr T}(\xi)=\sum_{\lambda \in \Lambda ^{\perp}} \widehat \eta_2(\xi + \lambda) \overline{\widehat{\eta}_1 (\xi + \lambda)}=0.$
		Similarly,    $\mc E^{\Lambda}(\eta_1)$  is  also an $\mc S^{\Lambda}(\eta_2)$-subspace orthogonal to $\mc E^{\Lambda}(\eta_2).$ 
	\end{enumerate}
\end{exa}

\begin{exa}  First we recall Example \ref{ex:discrete-dual2}  and also fix an automorphism $A$ on  $\widehat {\mc G}$  such that $A \Omega \subsetneq \Omega.$  Let   $ \mc A=\{\eta_1, \eta_2\}\subset L^2(\mc G)$ be such that   $\widehat \eta_1=
	\chi_{A \Omega}$ and $\widehat \eta_2=
	\chi_{\Omega \backslash  A\Omega}.$ Then the associated Gramian matrix  $G_{\mc A}(\xi) $  is $\begin{bmatrix} 1&0\\
	0&0
	\end{bmatrix}$
	for $\xi \in A \Omega,$  {and} 	$ \begin{bmatrix}
	0&0\\
	0&1
	\end{bmatrix}$ for $\xi \in \Omega \backslash  A\Omega$ by noting 
	$\sum_{\lambda \in \Lambda ^{\perp}}|\widehat \eta_1(\xi + \lambda)|^2= \chi_{A\Omega} (\xi),$ $  \sum_{k \in \Lambda ^{\perp}}|\widehat \eta_2 (\xi + \lambda)|^2= \chi_{\Omega \backslash  A\Omega} (\xi),$   and $ \sum_{\lambda \in \Lambda ^{\perp}} \widehat \eta_1(\xi + \lambda) \overline{\widehat{\eta}_2 (\xi + \lambda)}=0.$ 
	Further, let    $\mc A'=\{ \zeta_1, \zeta_2\}\in L^2(\mc G)$ be such that   $\widehat \zeta_1=
	\chi_{ \Omega\backslash A\Omega}$ and $\widehat \zeta_2=
	\chi_{A\Omega }$ 	for a.e. $\xi \in \Omega.$ Then the associated Gramian matrix  $G_{\mc A'}(\xi) $  is
	$ \begin{bmatrix}
	0&0\\
	0&1
	\end{bmatrix}$ 
	for $\xi \in A \Omega,$  {and} 
	$\begin{bmatrix} 1&0\\
	0&0
	\end{bmatrix}$ for $\xi \in \Omega \backslash  A\Omega.$
	Since  $G_{\mc A}(\xi)G_{\mc A'}(\xi)=0$ for a.e.  $\xi \in \Omega,$  the systems $\mc E^\Lambda(\mc A)$ and $\mc E^\Lambda(\mc A')$ are orthogonal pair   by Theorem \ref{T:Orthogonal}. From the Corollary \ref{Cor: Multi-dual}, note that $\mc E^\Lambda(\mc A)$ and $\mc E^\Lambda(\mc A')$ are $\mc S^\Lambda(\mc A)$ and $\mc S^\Lambda(\mc A')$-subspace dual to itself, respectively.  
\end{exa}
\begin{exa}\label{ex:discrete-dual1}  Let $\psi \in L^2 (\G)$ be such that $\mc E^{\Gamma}(\psi)$ is an $\mc S^{\Gamma}(\psi)$-subspace frame. Assume  that  $\varphi,\tilde{\varphi}\in L^2( \mathscr G)$  which are defined  in terms of the  Zak transform as follows,  for a.e. $\alpha \in \widehat \Gamma$
	$$
	(\mc Z\varphi)(\alpha, \Gamma x)=\mf m(\alpha) (\mc Z\psi)(\alpha, \Gamma x)\ 
	\mbox{and}\ (\mc Z\tilde\varphi)(\alpha, \Gamma x)=\tilde{\mf m}(\alpha)(\mc Z\psi)(\alpha, \Gamma x) \ \mbox{for all} \ \Gamma x\in \Gamma \backslash \mathscr G,
	$$
	where $\mf m,  \tilde{\mf m} \in L^2 \left(\widehat \Gamma, [ \psi, \psi]\right)$  are bounded functions.
	Then $\mc E^{\Gamma}(\varphi)$  is  an $\mc S^{\Gamma}(\varphi)$-subspace frame as $\mc S^{\Gamma}(\varphi)=\mc S^{\Gamma}(\psi)$ and $[ \varphi,  \varphi] (\alpha)=|\mf m(\alpha)|^2 [ \psi,  \psi](\alpha)$ for a.e.  $\alpha \in \widehat{\Gamma}.$ Similarly,     $\mc E^{\Gamma}(\tilde{\varphi})$ is also an   $\mc S^{\Gamma}(\tilde\varphi)$-subspace frame.  
	Also note that $[\varphi,\tilde \varphi](\alpha)=\mf m(\alpha) \ol{\tilde{\mf m}(\alpha)} [\psi,\psi](\alpha) $ for a.e.  $\alpha \in \widehat{\Gamma}.$  By the Corollary \ref{C:DualMain},  $\mc E^{\Gamma}(\varphi)$ is an $ \mc S^{\Gamma}(\tilde\varphi)$-subspace dual to $\mc E^{\Gamma}(\tilde{\varphi})$ 
	if and only if
	$$\mf m(\alpha)\ol{\tilde{ \mf m}(\alpha)}=\frac{1}{[\psi, \psi](\alpha)}\mbox{ on}\ \left\{\alpha\in \widehat \Gamma:[\psi, \psi](\alpha)\neq 0\right\}.
	$$
	In this case,  both $\mc E^{\Gamma}(\varphi)$ and   $\mc E^{\Gamma}(\tilde{\varphi})$  are dual frames to each other. The condition  gives various choices of subspace dual frames. 
	Also, $\mc E^{\Gamma}(\varphi)$ is an $ \mc S^{\Gamma}(\tilde\varphi)$-subspace orthogonal  to $\mc E^{\Gamma}(\tilde{\varphi})$ 
	if and only if $\mf m(\alpha)\ol{\tilde{ \mf m}(\alpha)}=0,$  a.e. $\alpha \in \widehat \Gamma.$  In this case   $\mc E^{\Gamma}(\varphi)$ and   $\mc E^{\Gamma}(\tilde{\varphi})$  are orthogonal pair.   
\end{exa}
The Example \ref{ex:discrete-dual1} also concludes that
 there is  a unique  $ \varphi \in \mc S^{\Gamma}(\psi)$  such that $\mc E^{\Gamma}( \varphi)$ is an $\mc S^{\Gamma}(\psi)$-subspace dual frame to $\mc E^{\Gamma}(\psi),$ where  
 	${\varphi}=\mc S^{\dagger}\psi$ (pseudo inverse).
 Since   $ 	[\psi,  \varphi](\alpha)=\ol{{\mf m}(\alpha)}[\psi, \psi](\alpha)=1 \ \mbox{a.e. on}\ \Omega_{ \psi},$
 	the function ${\mf m}$ is unique except on the set $\{\alpha \in \widehat \Gamma:[ \psi, \psi](\alpha)=0\}$ as various choices of ${\mf m}$ is possible on this set. But note that  whatever choice we can do for ${\mf m}$  we always have  $\mc Z\varphi (\alpha, \Gamma x)={\mf m} (\alpha) \mc Z\psi (\alpha, \Gamma x)=0$  on  $\left\{\alpha \in \widehat \Gamma:[ \psi,  \psi](\alpha)=0\right\}.$ So $\mc Z{\varphi}$ is uniquely defined, and  hence $ \varphi$ is  so.

The next result  discusses  an existence of $\mc S^{\Gamma}(\varphi)$-subspace dual  to  a frame $\mc E^{\Gamma}(\varphi)$ and provide a condition to get unique dual (upto a scalar multiplication). This generalizes a result described  for $L^2(\mathbb R^n)$ \cite[Theorem 4.3]{christensen2004oblique}.

\begin{thm} \label{Th:dual existence}	Let  $\varphi, \psi \in L^2(\mathscr G)$  be such that   $\mc E^{\Gamma}(\varphi)$ and $\mc E^{\Gamma}(\psi)$ are $\mc S^{\Gamma}(\varphi)$ and $\mc S^{\Gamma}(\psi)$-subspace frames, respectively. If  for some  positive constant $C,$ the expression  
	$\left|[ \varphi,  \psi](\alpha)\right| \geq C$   holds  for a.e. $\alpha \in \Omega_{ \varphi}$ (defined in Corollary \ref{C:DualMain}), 
	then	there exists a   $\tilde\psi\in \mc S^{\Gamma}(\psi)$ such that  $\mc E^{\Gamma}(\tilde\psi)$  is an    $\mc S^{\Gamma}(\varphi)$-subspace dual to $\mc E^{\Gamma}(\varphi).$   

Moreover,   $\tilde\psi\in \mc S^{\Gamma}(\psi)$ is unique if and only if $\Omega_{\varphi}=\Omega_{ \psi},$ a.e. 
	In particular,  the   $\tilde\psi\in \mc S^{\Gamma}(\psi)$  is unique and satisfies the following relation for  \mbox{a.e.}  $\alpha \in \Omega_{\varphi}$:
	$$
	[ \varphi, \psi](\alpha)	\, (\mc Z \tilde{\psi})(\alpha, \Gamma x)=(\mc Z \psi)(\alpha, \Gamma x) \, \chi_{\Omega_{ \varphi}} (\alpha)  \ \mbox{for all} \ \Gamma x \in \Gamma \backslash \mathscr G.
	$$
\end{thm}
\begin{proof}  Firstly note that for a.e. $\alpha \in \Omega_{\varphi},$ $C \leq |[\varphi, \psi](\alpha)| \leq \sqrt{BB'},$    follows  from the estimate (\ref{eq:cauchyschwartzbracket})
	\noindent  as $\mc E^{\Gamma}(\varphi)$ and $\mc E^{\Gamma}(\psi)$ are Bessel sequences  with  bounds $B$ and $B',$ respectively \cite{iverson2015subspaces}. Further note that any function $\tilde\psi\in \mc S^{\Gamma}(\psi)$ if and only if  for  a.e. $\alpha \in \widehat{\Gamma},$ we have $\mc Z\tilde{\psi}(\alpha, \Gamma x)=\mf m(\alpha)\mc Z \psi(\alpha, \Gamma x)$  for all $\Gamma x \in \Gamma \backslash \mathscr G,$ where  $\mf m\in L^2(\widehat \Gamma, [ \psi,  \psi]),$ from  Proposition \ref{P:decomposition}. From Corollary \ref{C:DualMain}, additionally note that 	$\mc E^{\Gamma}(\tilde \psi)$ is an $\mc S^{\Gamma}(\varphi)$-subspace dual to $\mc E^{\Gamma}(\varphi)$ if and only if   for a.e. $\alpha \in  \Omega_{ \varphi},$  
	$ 	1=[\varphi, { \tilde\psi}](\alpha)=\ol{\mf m(\alpha)}[\varphi,  \psi](\alpha). $
	Hence $\mf m$ is both bounded above and bounded below on $\Omega_{\varphi},$ extending  this to an arbitrary function in $L^2(\widehat \Gamma),$ will produce a function $\tilde \psi\in \mc S^\Gamma(\psi)$ such that  $\mc E^\Gamma(\tilde \psi)$ is an $\mc S^\Gamma(\varphi)$-subspace dual to $\mc E^\Gamma(\varphi).$
	
	\noindent Since   $\left|[ \varphi,  \psi](\alpha)\right| \geq C$ for a.e. $\alpha\in  \Omega_{ \varphi},$  we have     $\left\{\alpha\in \widehat \Gamma:[ \psi,  \psi](\alpha)=0\right\} \subseteq \left\{\alpha\in \widehat \Gamma:[ \varphi,  \varphi](\alpha) =0\right\}.$
	When the equality holds on the  above sets, we get 	 
	$\mc Z\tilde \psi(\alpha)=0$ on  $\left\{\alpha\in \widehat \Gamma:[ \varphi,  \varphi](\alpha)=0\right\}$ always, whatever $\mf m$   is  considered. In this case there exists unique $\tilde \psi,$ which fulfils the requirements.  Otherwise for the case  $\left\{\alpha\in \widehat \Gamma:[ \psi,  \psi](\alpha)=0\}\subsetneq \{\alpha\in \widehat \Gamma:[ \varphi,  \varphi](\alpha)(\alpha)=0\right\},$ various choices of  will leads to various $\mc Z\tilde \psi(\alpha)$  as it is non-zero,  by considering  $\mf m$ on $\left\{\alpha\in \widehat \Gamma:[ \varphi, \varphi](\alpha) (\alpha)=0\right\} \backslash \left\{\alpha\in \widehat \Gamma:[ \psi,  \psi](\alpha)=0\right\}.$  Thus the  various $\tilde \psi$ is possible.	Hence the result follows. 
\end{proof}


In this section we have discussed  $\mc S^\Gamma (\mc A)$-subspace duals of a frame/Bessel family $\mc E^\Gamma (\mc A)$ in $L^2 (\G),$ and realized that we can   obtain various duals of a  frame/Bessel family (see Example  \ref{ex:discrete-dual1} ). Theorem \ref{Th:dual existence} motivates to discuss about unique dual. It is well known that the unique dual can be   obtained when  the  frame/Bessel family $\mc E^\Gamma (\mc A)$ becomes Riesz basis for $L^2 (\G),$ known as \textit{dual basis} or \textit{biorthogonal basis}. We refer  \cite{christensen2016introduction, walnut2013introduction} for more details on Riesz basis and biorthogonal basis.   Next we study     biorthogonal systems and Riesz basis generated by translations in    $L^2 (\G).$  

 \section{Translation generated Biorthogonal system and subspace Riesz basis }\label{s:biortho}  For two non-zero functions $\varphi,\psi\in L^2(\G),$ the $\Gamma$-TG systems   $\mc E^{\Gamma}(\varphi)$ and  $\mc E^{\Gamma}(\psi)$ are said to be \textit{biorthogonal} if $\langle L_\gamma \varphi, L_{\gamma'}\psi\rangle=\delta_{\gamma,\gamma'}$  for all $\gamma,\gamma'\in \Gamma.$ Throughout the section, we assume $\Gamma$ to be discrete abelian subgroup of $\G.$ The following result characterizes biorthogonal systems in terms of the Zak transform and describes when a  translation generated system will be linearly independent.   $\mc E^{\Gamma}(\varphi)$  is \textit{linearly independent} if  the expression $\sum_{\gamma \in  \Gamma} c_\gamma L_\gamma \varphi = 0$ for some   $\{c_\gamma\}_{\gamma\in \Gamma} \in \ell^2 (\Gamma)$  implies $c_\gamma=0$ for all $\gamma.$

\begin{thm}\label{T:Bio1} 
	 For  non-zero functions $\varphi,\psi\in L^2(\G),$  the $\Gamma$-TG systems $\mc E^{\Gamma}(\varphi)$ and  $\mc E^{\Gamma}(\psi)$ are     biorthogonal if and only if  for a.e. $\alpha \in \widehat \Gamma,$ 
	$[ \varphi,  \psi](\alpha)=1.$  In case of either  one holds, then 
\begin{enumerate}
\item[(i)]  The  systems $\mc E^{\Gamma}(\varphi)$ and  $\mc E^{\Gamma}(\psi)$ are linearly independent.
\item[(ii)]  $\mc E^{\Gamma}(\psi)$ is an $\mc S^{\Gamma}(\varphi)$-subspace dual to   $\mc E^{\Gamma}(\varphi)$ for compactly supported $\varphi$ and $\psi.$  
\end{enumerate}  
\end{thm} 
 \begin{proof} Firstly  observe that the biorthogonal relation between the $\Gamma$-TG systems   $\mc E^{\Gamma}(\varphi)$ and  $\mc E^{\Gamma}(\psi)$  is equivalent to  $\langle L_\gamma \varphi, \psi\rangle=\delta_{\gamma,0}$ for $\gamma \in \Gamma.$  Now
{\small\begin{align*}
	\delta_{\gamma,0}=\langle \mc ZL_\gamma \varphi, \mc Z \psi \rangle_{L^2(\widehat 
		\Gamma;L^2( \Gamma\backslash \G))}
	=\int_{\widehat\Gamma}\int_{\Gamma \backslash \G}\mc Z\varphi(\alpha, \Gamma x)\ol{\mc Z\psi(\alpha, \Gamma x)\alpha( \gamma)}\ \dGammax \ \dalpha
	=\int_{\widehat \Gamma }[\varphi,\psi](\alpha)\ol{\alpha(\gamma)}\ \dalpha. 
\end{align*}}Thus the result follows by   the uniqueness of the Fourier coefficients.

 For the remaining part of (i), let  $\{c_\gamma\}_{\gamma\in \Gamma} \in \ell^2 (\Gamma)$ be such that  
$\sum_{\gamma \in  \Gamma} c_\gamma L_\gamma \varphi = 0.$ Then  for each $\gamma'\in \Gamma,$  $0=\langle 0, L_{\gamma'}\psi\rangle=\langle\sum_{\gamma \in  \Gamma}c_\gamma L_\gamma \varphi,L_{\gamma'} \psi \rangle=\sum_{ \gamma \in \Gamma} c_\gamma\langle L_\gamma \varphi, L_{\gamma'} \psi \rangle=c_{\gamma'} $
by the biorthogonal relation  between  $\mc  E^{\Gamma}(\varphi)$ and $\mc  E^{\Gamma}(\psi),$ and hence all  $c_\gamma$'s  are zero. Thus $\mc  E^{\Gamma}(\varphi)$ is linearly independent. Similarly,   $\mc  E^{\Gamma}(\psi)$ is also linearly independent.

\noindent(ii) Due to  $\langle\varphi, L_\gamma \psi\rangle=\delta_{\gamma, 0}$ for $\gamma \in \Gamma,$ we have  $f=\sum_{\gamma \in \Gamma_1}\langle f,L_{\gamma}\psi\rangle L_{\gamma}\varphi$  for  all $f \in \mbox{span} \ \mc E^\Gamma(\varphi)$ and   $\Gamma_1$ is a finite subset of $\Gamma,$  which holds for all $f \in \mc S^\Gamma(\varphi)$ in view of  compactly supported functions $\varphi$ and $\psi,$  and the continuity of  the function $f\mapsto \sum_{\gamma \in \Gamma}\langle f,L_{\gamma}\psi\rangle L_{\gamma}\varphi.$  Note that  $\mc E^\Gamma(\varphi)$ and $\mc E^\Gamma(\psi)$ are Bessel families since      $\{[ \varphi,  \varphi](\alpha)\}_{\alpha \in \widehat \Gamma}$ and $\{[ \psi,  \psi](\alpha)\}_{\alpha \in \widehat \Gamma}$ are bounded sets for a.e. $\alpha \in \widehat \Gamma$ \cite{iverson2015subspaces}. The boundedness of $\{[ \varphi,  \varphi](\alpha)\}_{\alpha \in \widehat \Gamma}$  follows by observing  the continuity of the function $\alpha \mapsto [ \varphi,  \varphi](\alpha)$ from the compact set $\widehat \Gamma$ to $\mathbb R.$ Indeed $[ \varphi,  \varphi](\alpha)$ is a polynomial for a.e. $\alpha \in \widehat \Gamma,$ that can be realised by writing it in the form of Fourier series expansion where   only finitely many  Fourier coefficients are non-zero in view of the compact support of $\varphi.$  
\end{proof}
\begin{cor}\label{Co:BioExis}
		Let $\varphi \in L^2 (\mathscr G)$ be such that for a.e.  $\alpha \in  \widehat 
		\Gamma,$    $C \leq [ \varphi,  \varphi](\alpha)  \leq D$  for some constants $0<C \leq D < \infty.$ Then there is a $\psi\in L^2 (\mathscr G)$ such that $\mc E^{\Gamma}(\varphi)$ and $\mc E^{\Gamma}(\psi)$ are biorthogonal systems.  Moreover, $\mc  E^{\Gamma} (\varphi)$ is linearly independent. 
		\end{cor}
\begin{proof} For $\varphi \in L^2 (\mathscr G),$  choose $\psi \in L^2 (\mathscr G)$ satisfying   $\mc Z\varphi(\alpha, \Gamma x) =\mc Z \psi(\alpha, \Gamma x) [\varphi, \varphi](\alpha)$ for a.e. $\alpha \in \widehat{\Gamma}$ and $\Gamma x \in \Gamma \backslash \G.$ Then  $\mc E^{\Gamma}(\varphi)$ and $\mc E^{\Gamma}(\psi)$ are biorthogonal systems in view of Theorem \ref{T:Bio1} since  $[\varphi, \psi](\alpha)=1$  for a.e. $\alpha \in \widehat \Gamma.$  The moreover part  also follows by  Theorem \ref{T:Bio1}. 
\end{proof}

Our next goal is to fix a function $\varphi \in L^2 (\G)$ and want to find $\psi \in  L^2 (\G)$ such that $\mc E^{\Gamma}(\psi)$  is an $\mc S^{\Gamma}(\varphi)$-subspace dual to $\mc E^{\Gamma}(\varphi)$  by  following the idea of  Theorem \ref{T:Bio1}.      To find such $\psi,$    we will     assume   $\varphi \in L^2 (\G)$ with  compact support such that $\mc E^{\Gamma}(\varphi)$  is an $\mc S^{\Gamma}(\varphi)$-subspace Riesz basis in the next result. 
 By an \textit{$\mc S^{\Gamma}(\varphi)$-subspace Riesz basis}, we mean   $\mc E^{\Gamma}(\varphi)$  is an $\mc S^{\Gamma}(\varphi)$-subspace frame and $\mc E^{\Gamma}(\varphi)$ is linearly independent. Equivalently, there are $0 <A \leq B <\infty$ such that  $A\sum_{ \gamma \in \Gamma}|c_\gamma|^2\leq\|\sum_{ \gamma \in \Gamma}c_\gamma L_\gamma\varphi\|^2\leq B\sum_{\gamma\in \Gamma}|c_{\gamma}|^2$ for some  sequence $\{c_\gamma\}_{\gamma\in \Gamma} \in \ell^2 (\Gamma)$ having finitely many non-zero  terms. We refer  \cite{bownik2015structure,iverson2015subspaces,walnut2013introduction} for more details.

\begin{thm}\label{T:Bio2}
	Let $\varphi \in L^2(\G)$ be a function with compact support  such that $\mc E^{\Gamma}(\varphi)$  is an $\mc S^{\Gamma}(\varphi)$-subspace Riesz basis.   If there exists a function $\psi \in L^2(\G)$ such that   $\mc  E^{\Gamma}(\psi)$ is biorthogonal to $\mc  E^{\Gamma}(\varphi),$ $\mc  E^{\Gamma}(\psi)$ is an $S^\Gamma(\varphi)$-subspace dual to $\mc  E^{\Gamma}(\varphi).$ Moreover, $\mc E^{\Gamma}(\psi)$  is also an $\mc S^{\Gamma}(\psi)$-subspace Riesz basis.
\end{thm}
\begin{proof} Firstly note that  for any $f\in \mbox{span}\  \mc  E^{\Gamma}(\varphi),$ we can write $f=\sum_{\gamma\in \Gamma'}\langle f,  L_\gamma \psi\rangle L_\gamma\varphi$ due to the biorthogonality of   $\mc  E^{\Gamma}(\psi)$  and  $\mc  E^{\Gamma}(\varphi)$ for some finite set  $\Gamma'$ in   $\Gamma.$ We need to show the expression for  all $f \in \mc  S^{\Gamma}(\varphi).$ For this let 
   $f \in  \mc S^{\Gamma}(\varphi).$ Then there is   an element $g\in \mbox{span}\  \mc E^{\Gamma}(\varphi) $  such that    $\|f-g\|< \epsilon$ for $\epsilon >0.$ By writing 
$g=\sum_{\gamma\in \Gamma_1}\langle f, L_\gamma \psi \rangle L_\gamma \varphi$ where $\Gamma_1$ is a finite subset of $\Gamma,$ we have $ f-\sum_{ \gamma\in \Gamma_1}\langle f, L_\gamma \psi \rangle L_\gamma \varphi= (f-g) +\sum_{\gamma\in \Gamma_1} \langle (g-f), L_\gamma \psi \rangle L_\gamma \varphi,$  and by taking norm on the both sides, we obtain
\begin{align*}
\|f-\sum_{ \gamma\in \Gamma_1}\langle f, L_\gamma \psi\rangle L_\gamma \varphi\|&\leq\|f-g\|+\|\sum_{ \gamma\in \Gamma_1} \langle (g-f), L_\gamma \psi \rangle L_\gamma \varphi\|\\
&\leq\|f-g\|+ \sqrt{B} \Big( \sum_{ \gamma\in \Gamma_1} |\langle (g-f), L_\gamma \psi \rangle|^2 \Big)^{1/2}\\
&\leq (1+ \sqrt{B'B})\|f-g\|  < (1+ \sqrt{B'B}) \epsilon  
\end{align*}
for some $B, B'>0$  since $\mc E^{\Gamma}(\varphi)$  is an $\mc S^{\Gamma}(\varphi)$-subspace Riesz basis.  The last inequality holds true provided $\sum_{ \gamma \in \Gamma}|\langle f, L_{\gamma}\psi\rangle|^2\leq B'\|f\|^2$ for all $f \in \mc S^{\Gamma}(\varphi).$   For this,  let $f \in \mc S^{\Gamma}(\varphi).$ Then there is a  sequence $\{f_n\}_{n\in \mathbb N}$ in $\mbox{span} \, \mc E^{\Gamma}(\varphi)$ such that $\lim\limits_{n \ra \infty}$ $\|f_n-f\|=0,$   and also by Cauchy-Schwarz inequality, we have $\lim\limits_{n \ra \infty}$$\langle f_n, L_\gamma\psi\rangle=\langle f, L_\gamma\psi\rangle$   for every $\gamma \in \Gamma.$ Hence for any finite set $\Gamma_1$ of $\Gamma,$ we have 
\begin{align*}
&\sum_{ \gamma \in \Gamma_1}|\langle f, L_{\gamma}\psi\rangle|^2=\sum_{\gamma\in \Gamma_1} \lim\limits_{n \ra \infty} |\langle f_n, L_{\gamma}\psi\rangle|^2=\lim\limits_{n \ra \infty}\sum_{ \gamma \in \Gamma_1}|\langle f_n, L_{\gamma}\psi\rangle|^2\leq B' \lim \limits_{n \ra \infty}\|f_n\|^2= B'\|f\|^2,  
\end{align*}
provided $\sum_{ \gamma \in \Gamma}|\langle f, L_{\gamma}\psi\rangle|^2\leq B'\|f\|^2$ for all $f \in \mbox{span} \, \mc E^{\Gamma}(\varphi).$  For this we proceed as follows.\\
 By writing $f \in \mbox{span} \, \mc E^{\Gamma}(\varphi)$   in the form  $f=\sum_{ \gamma \in \Gamma}c_{\gamma}L_\gamma\varphi$ with finitely many non-zeros   $\{c_\gamma\}_{\gamma\in \Gamma}\in \ell^2,$ we get 	$\|f\|^2=\|\mc Zf\|^2	=\int_{\widehat \Gamma}| \widehat c(\alpha)|^2 [ \varphi,  \varphi](\alpha) \ \dalpha \nonumber$ by following the steps of (\ref{eq:normf}). Since $\mc E^{\Gamma}(\varphi)$ is an  $\mc S^\Gamma(\varphi)$-subspace Riesz basis,  we have  $C\leq [\varphi, \varphi](\alpha)\leq D$  for a.e. $\alpha\in \widehat \Gamma$ and  some $0< C \leq D < \infty$ \cite[Remark 5.6]{iverson2015subspaces}, and hence $C\int_{ \widehat \Gamma}|\widehat{c}(\alpha)|^2\dalpha \leq \|f\|^2\leq D \int_{\widehat \Gamma}|\widehat{c}(\gamma)|^2\ \dalpha.$ Thus, we get  
$ \int_{ \widehat \Gamma}|\widehat{c}(\alpha)|^2\ \dalpha \leq \frac{1}{C} \|f\|^2.$ Further due to the biorthogonality of the sets $\mc E^{\Gamma}(\varphi)$ and $\mc E^{\Gamma}(\psi)$ and Parseval's formula, we write 
$\int_{ \widehat \Gamma }|\widehat{c}(\alpha)|^2\ \dalpha=\sum_{\gamma\in \Gamma}|c_{\gamma}|^2=\sum_{ \gamma \in \Gamma}|\langle f, L_{\gamma}\psi\rangle|^2,$ which gives the    inequality
$\sum_{ \gamma \in \Gamma}|\langle f, L_{\gamma}\psi\rangle|^2\leq B'\|f\|^2$ for  $f \in \mbox{span} \, \mc E^{\Gamma}(\varphi)$ and $B'=\frac{1}{C}.$ Thus the result follows. 
\end{proof}
 In the following example we  construct various biorthogonal systems using Theorem \ref{T:Bio1}.
\begin{exa}  First we recall Example \ref{ex:discrete-dual1} and assume that  $\mc E^{\Gamma}(\psi)$ is an $\mc S^{\Gamma}(\psi)$-subspace Riesz basis.   The functions   $\varphi,\tilde{\varphi}\in L^2( \mathscr G)$    are defined  by $(\mc Z\varphi)(\alpha, \Gamma x)=\mf m(\alpha) (\mc Z\psi)(\alpha, \Gamma x)$ and $(\mc Z\tilde\varphi)(\alpha, \Gamma x)=\tilde{\mf m}(\alpha)(\mc Z\psi)(\alpha, \Gamma x)$   for all $\Gamma x\in \Gamma \backslash \mathscr G,$ and for a.e. $\alpha \in \widehat \Gamma,$  where  $\mf m,  \tilde{\mf m} \in L^2 \left(\widehat \Gamma, [ \psi,  \psi]\right).$ 	Then in view of Theorem \ref{T:Bio1},  $\mc E^{\Gamma}(\varphi)$  and $ \mc E^{\Gamma}(\tilde\varphi)$ are biorthogonal  
	if and only if
	$\mf m(\alpha)\ol{\tilde{ \mf m}(\alpha)}=\frac{1}{[\psi,\psi](\alpha)}$ for a.e. $\alpha \in \widehat \Gamma,$ follows from the calculations of Example \ref{ex:discrete-dual1}.

\end{exa}
\section{Orbit generated by  the action of a non-discrete abelian subgroup} \label{s:non-discrete}
The purpose of this section is devoted to characterize a pair of orthogonal frames and  subspace dual  of a Bessel family/frame   generated by  the $\Gamma$-TG system $\mc E^\Gamma(\mc A)=\{L_\gamma\varphi_t:\gamma \in \Gamma, t \in \mc N\}$  in $L^2(\G),$  where the group   $\Gamma$ is a closed abelian  (need not be discrete) subgroup of $\G,$   $\mc N$ is a  $\sigma$-finite measure space (need not be countable), and $\mc A =\{\varphi_t\}_{t\in \mc N} \subset L^2 (\G).$ We characterize such results using the  Zak transform $\mc Z$ for the pair $(\mathscr G, \Gamma)$ defined by (\ref{eq:ZakTransform}). When  $\mathscr G$ becomes an abelian group $\mc G,$  the fiberization map is also used which unifies the classical results related to the orthogonal and duals of a Bessel family/frame associated with a TI  space.  
\subsection{Orthogonal and dual frames using range function}
Due to the discrete nature of $\Gamma,$ we made use of the orthonormal basis property of $\widehat \Gamma$ in the  previous section. However, in this section, we are making use of the approach of range function on $\widehat \Gamma.$  A  \textit{range function}  on $\widehat\Gamma$ is a mapping 
$J$ from $\widehat\Gamma$ to $\{\text{closed subspaces of} ~L^2(\Gamma\backslash \G)\},$ which   is  measurable  if   
$\alpha\mapsto \langle P_J(\alpha)u,v\rangle$ is measurable on $\widehat \Gamma$ for any $u,v \in L^2(\Gamma\backslash \G),$ where for a.e. $\alpha \in \widehat \Gamma,$ $P_{J} (\alpha): L^2(\Gamma\backslash \G) \rightarrow  J(\alpha)$ is an  orthogonal projection.  The associated    closed subspace $V_J$ of $L^2(\Gamma\backslash \G)$  defined by  $V_J =\left\{\varphi \in L^2(\G):\mc Z\varphi(\alpha)\in J(\alpha)~\mbox{for}~ \a.e.~ \alpha\in \widehat \Gamma\right\}$  plays a crucial  role to establish the theory.  Further  it can be  noted that a  $\Gamma$-TI subspace in $L^2(\G)$ can be identified with  a range function $J.$ Indeed there is a bijection  $J\mapsto V_J.$ In particular,  the range function   $J(\alpha)= \overline{\mbox{span}}\{(\mc Zf)(\alpha) : f \in \mc A_0\} $ is associated with  $V_J=\mc S^\Gamma(\mc A)$ for some countable dense subset $\mc A_0$ of $\mc A$ in $L^2 (\mathscr G).$ For more details about the range function,  we refer \cite{bownik2015structure, iverson2015subspaces}.

 Now, we are going to talk about our main  result, which is connected to the orthogonal subspace and dual subspace of a Bessel family that is associated with the range function in terms of the Zak transform. It includes certain results of \cite[Corrollary 5.11, Corrollary 5.13]{bownik2019multiplication}, which contains an alternative strategy for proving the result.

\begin{thm}\label{Th: Dual-Ortho-Non Discerete}  Let $(\mc N,\mu_{\mc N})$  be  a complete,   $\sigma$-finite measure space and let $\mc A=\{\varphi_t\}_{t\in \mc N}$ and $\mc A'=\{\psi_t\}_{t\in \mc N}$ be two collections of functions  in $L^2(\mathscr G)$ such that the $\Gamma$-TG systems $ \mc E^{\Gamma}(\mc A)$  and $\mc  E^{\Gamma}(\mc A')$ are Bessel. Assume  $\mc A$ has a countable dense subset $\mc A_0$ for which $J_{{\mc  A}} (\alpha)=  \overline{{\Span}}\{(\mc Zf)(\alpha) : f \in \mc A_0\}$    a.e.  $\alpha \in \widehat{\Gamma},$  where   ${\mc Z \mc A} = \{\mc Z f: f \in \mc A\}.$ Then  the following hold true: 
	\begin{itemize}
		\item[(i)] $\mc E^{\Gamma}(\mc A')$  is an $\mc S^{\Gamma}(\mc A) $-subspace   dual  to $\mc E^{\Gamma}(\mc A)$ if and only if 
	 	for a.e. $\alpha \in \widehat{\Gamma},$ the system $(\mc Z \mc A') (\alpha)=\{\mc Z \psi (\alpha): \psi \in \mc A'\}$ is a   $J_{ \mc A}  (\alpha)$-subspace dual  to $(\mc Z \mc A) (\alpha)=\{\mc Z \varphi (\alpha): \varphi \in \mc A\}.$ 
	 	\item[(ii)] $\mc E^{\Gamma}(\mc A')$  is an $\mc S^{\Gamma}(\mc A) $-subspace   orthogonal  to $\mc E^{\Gamma}(\mc A)$ if and only if 
	 	for a.e. $\alpha \in \widehat{\Gamma},$ the system $(\mc Z \mc A') (\alpha)$ is a   $J_{ \mc A}  (\alpha)$-subspace orthogonal  to $(\mc Z \mc A) (\alpha).$ 
	\end{itemize}
When  the pair $(\mathscr G, \Gamma)$ is   an abelian pair  $(\mc G, \Lambda),$ let 
$J_{\mc A} (\beta \Lambda^\perp)=  \overline{{\Span}}\{(\mathscr T f)(\beta \Lambda^\perp) : f \in \mc A_0\}$
 for a.e. $\beta \Lambda^\perp \in \widehat{\mc G}/\Lambda^\perp,$ where   ${\mathscr T \mc A} = \{\mathscr T f: f \in \mc A\}.$ Then (i) and (ii) become (i') and (ii') as follows:
 	\begin{itemize}
 	\item[(i')] $\mc E^{\Lambda}(\mc A')$  is an $\mc S^{\Lambda}(\mc A) $-subspace   dual  to $\mc E^{\Lambda}(\mc A)$ if and only if 
  for a.e. $\beta \Lambda^\perp \in \widehat{\mc G}/\Lambda^\perp,$  the system $(\mathscr T \mc A') (\beta \Lambda^\perp)=\{\mathscr T  \psi (\beta \Lambda^\perp): \psi \in \mc A'\}$ is a   $J_{\mc A}  (\beta \Lambda^\perp)$-subspace dual  to $(\mathscr T \mc A) (\beta \Lambda^\perp)=\{\mathscr T  \varphi (\beta \Lambda^\perp): \varphi \in \mc A\}.$
   
 	\item[(ii')] $\mc E^{\Lambda}(\mc A')$  is an $\mc S^{\Lambda}(\mc A) $-subspace   orthogonal  to $\mc E^{\Lambda}(\mc A)$ if and only if   	for a.e.  $\beta \Lambda^\perp \in \widehat{\mc G}/\Lambda^\perp,$ the system $(\mathscr T \mc A') (\beta \Lambda^\perp)$ is a   $J_{\mc A}  (\beta \Lambda^\perp)$-subspace orthogonal  to $(\mathscr T \mc A) (\beta \Lambda^\perp).$
 \end{itemize}
 \end{thm}
\noindent Before proceeding towards a proof of Theorem \ref{Th: Dual-Ortho-Non Discerete},  we first establish the following result to express the change of role of $\Gamma$ to $\widehat \Gamma$ in terms of  the  Zak transform $\mc Z,$ which basically a continuous version of Lemma \ref{L:BesselCharecterization}. To remove the confusion we provide a proof also.
\begin{prop}\label{P:BesselCharecterization} Assuming the   hypotheses of Theorem   \ref{Th: Dual-Ortho-Non Discerete},   the following holds 	for all $f,g \in L^2(\G)$:
	\begin{align*}
	\int_{\mc N}\int_{\Gamma}\left \langle f, L_{\gamma}\varphi_t\right \rangle\langle L_{\gamma}\psi_t,g\rangle\ d{\mu_{\mc N}}(t)\ d{\mu_{\Gamma}}(\gamma)
	=\int_{\mc N}\int_{\widehat {\Gamma}}\left\langle  \mc Zf(\alpha), \mc Z\varphi_t(\alpha)\right\rangle\left \langle \mc Z\psi_t(\alpha),\mc Zg(\alpha)\right \rangle \ d{\mu_{\widehat \Gamma}}(\alpha)\ d{\mu_{\mc N}}(t).
	\end{align*}
\end{prop}
\begin{proof}	Applying the Zak transform,
	\begin{align}\label{eq:Gamma to widehat Gamma}
\int_{\mc N}\int_{\Gamma}\left \langle f, L_{\gamma}\varphi_t\right  \rangle
&\langle L_{\gamma}\psi_t,g\rangle \ d{\mu_{\Gamma}}(\gamma)\ d{\mu_{\mc N}}(t) =\int_{\mc N}\int_{\Gamma} \langle \mc Zf, \mc Z(L_{\gamma}\varphi_t)\rangle \langle \mc Z(L_{\gamma}\psi_t), \mc Zg\rangle \ d{\mu_{\Gamma}}(\gamma)\ d{\mu_{\mc N}}(t) \nonumber\\
& =	\int_{\mc N}\int_{\Gamma} \left(\int_{\widehat \Gamma} \zeta_t(\alpha)   {\alpha(\gamma)} \ d{\mu_{\widehat \Gamma}}(\alpha)\right)\ol{\left(\int_{\widehat \Gamma} \eta_t(\alpha) \alpha(\gamma) \ d{\mu_{\widehat \Gamma}}(\alpha)\right)}\ d{\mu_{\Gamma}}(\gamma)\ d{\mu_\mc N}(t),
\end{align}
where $\zeta_t(\alpha)=\langle \mc Zf(\alpha),\mc Z\varphi_t(\alpha)\rangle$ and 
$\eta_t(\alpha)=\langle \mc Zg(\alpha),\mc Z\psi_t(\alpha)\rangle$ for each $t\in \mc N.$ The functions 	   $\zeta_t$ and $\eta_t$ are   in  $L^1(\widehat\Gamma)$ due to Cauchy-Schwarz inequality and  
\begin{align*}
&\int_{\widehat{\Gamma}}|\zeta_t(\alpha)|\ d{\mu_{\widehat\Gamma}}(\alpha)=\int_{\widehat{\Gamma}}\left|\int_{\Gamma\backslash \mathscr G}\mc Zf(\alpha)(\Gamma x) \ol{\mc Z\varphi_t(\alpha)(\Gamma x)} \ d{\mu_{\Gamma\backslash \mathscr G}}(\Gamma x)\right| \ d{\mu_{\widehat \Gamma}}(\alpha)\\
 &\leq \left(	\int_{\widehat{\Gamma}} \int_{\Gamma\backslash \mathscr G}|\mc Zf(\alpha)(\Gamma x)|^2 \ d{\mu_{\Gamma\backslash \mathscr G}}(\Gamma x)\ d{\mu_{\widehat \Gamma}}(\alpha) \right)^{1/2}
\left(\int_{\widehat{\Gamma}}\int_{\Gamma\backslash \mathscr G}|\mc Z\varphi_t(\alpha)(\Gamma x)|^2   \ d{\mu_{\Gamma\backslash \mathscr G}}(\Gamma x)  \dalpha \right)^{1/2}\\
&=\|\mc Z f\| \|\mc Z \varphi_t\|=\|f\|\|\varphi_t\|<\infty.
\end{align*}
Similarly,  $\eta_t \in L^1 (\widehat \Gamma).$  Then for each $t \in \mc N,$ the inverse Fourier transform $\check{\zeta_t}$ and $\check{\eta_t}$  are   members of $L^2 (\Gamma),$ where 
\begin{align*}
\check{\zeta_t}(\gamma)=\int_{\widehat\Gamma}\zeta_t(\alpha)   \alpha(\gamma)\ \dalpha \ \mbox{and} \ \check{\eta_t} (\gamma)=\int_{\widehat\Gamma}\eta_t(\alpha)  \alpha(\gamma) \  \dalpha.
\end{align*} 
This follows by   observing the Bessel property of $\mc E^{\Gamma}(\mc A)$ and calculations
\begin{align*}
&\infty>\int_{\mc N}\int_{\Gamma}\left | \langle f, L_{\gamma}\varphi_t\right \rangle |^2 \ d{\mu_{\Gamma}}(\gamma) 
=\int_{\mc N}\int_{\Gamma}\left | \int_{\widehat{\Gamma}}\langle \mc Zf(\alpha), \mc Z(L_{\gamma}\varphi_t)(\alpha) \rangle \ d{\mu_{\widehat \Gamma}}(\alpha)\right|^2 \  d{\mu_{\Gamma}}(\gamma) \ d{\mu_{\mc N}}(t)\\
&=\int_{\mc N}\int_{\Gamma}\left | \int_{\widehat{\Gamma}}\langle \mc Zf(\alpha), \mc Z\varphi_t(\alpha) \rangle \alpha(\gamma) \ d{\mu_{\widehat \Gamma}}(\alpha)\right|^2 \ d{\mu_{\Gamma}}(\gamma)\  d{\mu_{\mc N}}(t)
=\int_{\mc N}\int_{\Gamma}|\check{\zeta_t} (\gamma)|^2 \ d{\mu_{\Gamma}}(\gamma)\  d{\mu_{\mc N}}(t).
\end{align*}
Similarly, we have $\check{\eta_t} \in L^2 (\Gamma).$ Therefore, the equation (\ref{eq:Gamma to widehat Gamma}) is equal to the following  
 \begin{align*}
\int_{\mc N}\int_{\Gamma}\check{\zeta_t}(\gamma)\ol{\check{\eta_t}(\gamma)}\ d{\mu_{\Gamma}}(\gamma)\ d{\mu_{\mc N}}(t)
 =\int_{\mc N}\int_{ \widehat \Gamma }\eta_t(\alpha)\ol{\zeta_t(\alpha)} \ \dalpha\ d{\mu_{\mc N}}(t).
\end{align*}
Thus the result follows. 
\end{proof}
\begin{proof}  (i)  Firstly  assume $\mc E^{\Gamma}(\mc A')$  is an $\mc S^{\Gamma}(\mc A)$-subspace dual to $\mc E^\Gamma(\mc A).$ Then for $f\in \mc S^\Gamma(\mc A)$ and $g\in L^2(\G),$ we have  $
		 \int_{t\in\mc N}\int_{\gamma\in \Gamma}\langle f, L_\gamma\psi_t\rangle \langle L_\gamma\varphi_t,g\rangle\  d{\mu_{\Gamma}}(\gamma)\ d{\mu_{\mc N}}(t)= \langle f, g\rangle.  
$
Equivalently,  we get the following  by applying  Proposition \ref{P:BesselCharecterization} and  the  Zak transform $\mc Z$:
		\begin{align*}
\int_{\mc N}\int_{\widehat {\Gamma}}\left\langle  \mc Zf(\alpha), \mc Z\varphi_t(\alpha)\right\rangle\left \langle \mc Z\psi_t(\alpha),\mc Zg(\alpha)\right \rangle\ d{\mu_{\widehat \Gamma}}(\alpha)\ d{\mu_{\mc N}}(t)=\int_{\widehat \Gamma}\langle \mc Zf(\alpha),\mc Zg(\alpha)\rangle\ d{\mu_{\widehat \Gamma}}(\alpha). 
	\end{align*} 
	To get the result for a.e. $\alpha\in \widehat \Gamma,$  we   show $	\int_{\mc N}\left\langle  \mc Zf(\alpha), \mc Z\varphi_t(\alpha)\right\rangle\left \langle \mc Z\psi_t(\alpha),\mc Zg(\alpha)\right \rangle d{\mu_{\mc N}}(t)=\langle \mc Zf(\alpha),\mc Zg(\alpha)\rangle $  for 	$f\in \mc S^\Gamma(\mc A)$ and $ g\in L^2(\G).$  	On the contrary, we assume a Borel measurable subset $Y$ in  $\widehat{\Gamma}$ having positive measure such that the equality does not hold  on $Y.$  Then, there are $m_0, n_0 \in \mathbb N$ such that  $S_{m_0, n_0} \cap Y$ is a Borel measurable subset  of $\widehat{\Gamma}$ having positive measure, where  for each $m,n \in \mathbb N,$ the  set $S_{m, n}$ is 
	\begin{align*}
S_{m, n}=\left\{	\alpha\in \widehat \Gamma:\rho_{m, n} (\alpha):=\int_{\mc N}\left\langle P_{J_{\mc A}}(\alpha)x_m, \mc Z\varphi_t(\alpha)\right\rangle\left \langle \mc Z\psi_t(\alpha), x_n\right \rangle d{\mu_{\mc N}}(t)-\langle P_{J_{\mc A}}(\alpha)x_m,x_n\rangle \neq \{0\}\right\},
	\end{align*} $P_{J_{\mc A}} (\alpha)$ is  an orthogonal projection onto $J_{\mc A}(\alpha)$ for a.e. $\alpha\in \widehat \Gamma$ and $\{x_n\}_{n \in \mathbb N}$ is   a countable dense subset of $L^2(\Gamma \backslash \G).$
Clearly $\{P_{J_{\mc A}}  (\alpha) x_n\}_{n \in \mathbb N}$ is    dense in $J_{\mc A} (\alpha)$ for a.e. $\alpha \in \widehat \Gamma.$ 
	Hence   either real or imaginary parts of $\rho_{m_0, n_0}(\alpha)$ are strictly positive or negative for  a.e. $\alpha \in S_{m_0, n_0} \cap Y.$ By adopting the standard technique first we assume   the real part of $\rho_{m_0, n_0}(\alpha)$ is strictly positive on $S_{m_0, n_0} \cap Y.$ By choosing   a Borel measurable subset $S$ of $S_{m_0, n_0} \cap Y$ having positive measure, we define functions $h_1$ and $h_2$ as follows:  $h_1(\alpha)=\begin{cases} 
	P_{J_{\mc A}} (\alpha) x_{m_0} & \mbox{for } \alpha \in S,\\
	0 & \mbox{for } \alpha \in \widehat{\Gamma} \backslash S,
	\end{cases}	$ and $	h_2(\alpha)=\begin{cases} 
	P_{J_{\mc A}} (\alpha) x_{n_0} & \mbox{for } \alpha \in S,\\
	0 & \mbox{for } \alpha \in \widehat \Gamma \backslash S.
	\end{cases}$\\
	Then we have $h_1 (\alpha), h_2 (\alpha)\in J_{\mc A} (\alpha)$  for a.e. $\alpha \in \widehat{ \Gamma}$ since $\{P_{J_{\mc A}}  (\alpha) x_n\}_{n \in \mathbb N}$ is    dense in $J_{\mc A} (\alpha).$ Hence we get $h_1, h_2 \in \mc S^{\Gamma}(\mc A)$ which gives  	 	$\int_S   \rho_{m_0, n_0}(\alpha)    \    \dalpha =0.$ We  arrive on the contradiction since the measure of $S$ is positive and the real part of $\rho_{m_0, n_0}(\alpha)$ is strictly positive on $S.$ Other cases follow in a similar way.  Thus the result follows. 
	
	The converse part follows easily  using the Proposition \ref{P:BesselCharecterization}.

\noindent (ii) For $f\in \mc S^\Gamma(\mc A)$ and $g\in  L^2(\G),$  first assume  $\int_{t\in\mc N}\int_{\gamma\in \Gamma}\langle f, L_\gamma\psi_t\rangle \langle L_\gamma\varphi_t,g\rangle\ d{\mu_{\Gamma}}(\gamma)\  d{\mu_{\mc N}}(t)=0,$ which is equivalent to $
\int_{\mc N}\int_{\widehat {\Gamma}}\left\langle  \mc Zf(\alpha), \mc Z\varphi_t(\alpha)\right\rangle\left \langle \mc Z\psi_t(\alpha),\mc Zg(\alpha)\right \rangle\  d{\mu_{\widehat \Gamma}}(\alpha)\ d{\mu_{\mc N}}(t)=0,
$
from Proposition \ref{P:BesselCharecterization}. 
To get the result for a.e. $\alpha\in \widehat \Gamma,$  we   need to show $	\int_{\mc N}\left\langle  \mc Zf(\alpha), \mc Z\varphi_t(\alpha)\right\rangle\left \langle \mc Z\psi_t(\alpha),\mc Zg(\alpha)\right \rangle d{\mu_{\mc N}}(t)=0$   for 	$f\in \mc S^\Gamma(\mc A)$ and $ g\in L^2(\G).$ For this, let $(e_i)_{i\in \mathbb Z}$ be an  orthonormal basis for $L^2(\Gamma\backslash \G)$ and  $P_{J_{\mc A}} (\alpha)$ is  an orthogonal projection onto $J_{\mc A}(\alpha)$ for a.e. $\alpha\in \widehat \Gamma.$ Assume on the contrary,  there exists $i_0\in \mathbb Z$ such that 
$h(\alpha)=\int_{\mc N} \langle P(\alpha) e_{i_0}, \mc Z\varphi_t(\alpha)\rangle\ol{  \langle\mc Z\psi_t(\alpha) ,  \mc Zg (\alpha)\rangle}\ d{\mu_{\mc N}}(t)\neq 0$ on a measurable set $E\ss \widehat{\Gamma}$ with $\mu_{\widehat \Gamma}(E)>0.$
The rest of the proof follows in a similar   manner  of Theorem \ref{T:dualityMultiGenerators}  (ii). The converse part follows immediately by Proposition \ref{P:BesselCharecterization}.
	\end{proof}
For the case of  $\mc S^\Gamma(\mc A)=\mc S^\Gamma(\mc A')$ in Theorem \ref{Th: Dual-Ortho-Non Discerete}, we get   $J_{\mc A'(\alpha)}=J_{\mc A}(\alpha)$ a.e. $\alpha\in \widehat \Gamma,$  follows by observing the bijection of $J\mapsto V_J$ and $V_J=\mc S^\Gamma(\mc A)=\mc S^\Gamma(\mc A').$	 Then we have following result. 
\begin{cor} \label{T:dual-orthogonal-frame-nondiscrete}
	Under the   hypotheses mentioned in Theorem \ref{Th: Dual-Ortho-Non Discerete} let us assume $\mc S^\Gamma(\mc A)=\mc S^\Gamma(\mc A').$ 
	  \begin{enumerate}
		\item[(i)]$\mc E^{\Gamma}(\mc A')$  and   $\mc E^{\Gamma}(\mc A)$ are dual frame to each other  if and only if for a.e. $\alpha \in \widehat{\Gamma},$ the system $(\mc Z \mc A') (\alpha)=\{\mc Z\psi(\alpha):\psi\in\mc A\} $ and  $(\mc Z \mc A) (\alpha)=\{\mc Z\varphi(\alpha):\varphi\in\mc A\}$ are dual to each other. 
		\item[(ii)] $\mc E^{\Gamma}(\mc A')$  and   $\mc E^{\Gamma}(\mc A)$ are orthogonal pair  if and only if for a.e. $\alpha \in \widehat{\Gamma},$ the system $(\mc Z \mc A') (\alpha)$ and  $(\mc Z \mc A) (\alpha)$ are orthogonal pair.  
	\end{enumerate} 
\end{cor}

\subsection{Super dual frames}
Orthogonality is a fundamental idea that plays a significant role in the discussion of the dual frame related property of super-frames in orthogonal direct sum of Hilbert spaces. This concept was first presented by Han and Larson \cite{han2000frames} and Balan \cite{balan2000multiplexing}, who developed it further. This notion is further generalized  in the context of TI and Gabor systems  \cite{li2013super, lopez2013discrete}.
 By \textit{super Hilbert space} $L^2(\G)\oplus\dots\oplus L^2(\G)$ (N-copies), or $\oplus^N L^2(\G),$ we mean it is a collection of functions of the form   $\{\oplus_{n=1}^N f^{(n)} :=(f^{(1)}, f^{(2)}, \cdots,   f^{(N)}): f^{(n)}\in L^2(\G), 1 \leq n \leq N\}$ with the inner product $\langle \oplus_{n=1}^Nf^{(n)},\oplus_{n=1}^Ng^{(n)}\rangle=\sum_{n=1}^N \langle f^{(n)},g^{(n)}\rangle.$  Indeed,   	$\oplus^N L^2(\G)$ is nothing but the Hilbert space $L^2(\G \times \mathbb Z_N),$ where $\mathbb Z_N$ is an abelian group with modulo $N.$
Analogous to the classical trend, we state the following  characterization result for   dual frames of translates  in  the super Hilbert space $\oplus^N L^2(\G)$ (named as super dual frame).

\begin{thm}\label{T:Superdual}	Let $N\in \mathbb N$ and $\mc N$ be an $\sigma$-finite measure space with counting measure. For 	$1\leq n\leq N,$  let $\{\varphi_t^{(n)}\}_{t\in \mc N}$ and $\{\psi_t^{(n)}\}_{t\in \mc N}$ be two collections of functions in $L^2(\G)$  such   that	$\{L_\gamma \varphi_t^{(n)}\}_{t \in \mc N,\gamma \in \Gamma}$ and $\{L_\gamma \psi_t^{(n)}\}_{t\in \mc N, \gamma \in \Gamma}$ are Bessel. For each $\gamma\in \Gamma,$ define the translation operator $\mc L_\gamma:= \oplus^N L_\gamma $ which acts on  an element  $\oplus_{n=1}^N f^{(n)}$ by $\mc L_\gamma(\oplus_{n=1}^N f^{(n)})= \oplus_{n=1}^N L_\gamma f^{(n)}.$ Then  $\{\mc L_\gamma (\oplus_{n=1}^N \varphi_t^{(n)})\}_{t \in \mc N,\gamma \in \Gamma}$ and $\{\mc L_\gamma (\oplus_{n=1}^N \psi_t^{(n)})\}_{t \in \mc N,\gamma \in \Gamma}$ are   (super) dual frames     in $\oplus^N L^2(\G)$  if and only if  for a.e. $\alpha \in \widehat \Gamma,$  (i)     the systems 
	$\left\{\{\mc Z \varphi_t^{(n)} (\alpha, \Gamma x)\}_{\Gamma x \in \Gamma \backslash\G}\right\}_{t \in \mc N}$ and   	$\left\{\{\mc Z \psi_t^{(n)} (\alpha, \Gamma x)\}_{\Gamma x \in \Gamma \backslash\G}\right\}_{t \in \mc N}$ are dual frames  in $L^2 (\Gamma \backslash \G)$  for   $1\leq n\leq N,$ and   
	(ii) $\left\{\{\mc Z \varphi_t^{(n_1)} (\alpha, \Gamma x)\}_{\Gamma x \in \Gamma \backslash\G}\right\}_{t \in \mc N}$ and   	$\left\{\{\mc Z \psi_t^{(n_2)} (\alpha, \Gamma x)\}_{\Gamma x \in \Gamma \backslash\G} : t \in \mc N\right\}$ are orthogonal  pair for $1\leq n_1\neq n_2\leq N.$ 
\end{thm}
\begin{proof} Assume   the systems $\{\mc L_\gamma (\oplus_{n=1}^N \varphi_t^{(n)})\}_{t \in \mc N,\gamma \in \Gamma}$ and $\{\mc L_\gamma (\oplus_{n=1}^N \psi_t^{(n)})\}_{t \in \mc N,\gamma \in \Gamma}$ are   (super) dual frames     in $\oplus^N L^2(\G).$ Then for each $1\leq n \leq N,$ (i) follows by just applying the orthogonal projection $P_n$  on it  and Corollary  \ref{T:dual-orthogonal-frame-nondiscrete} .  For the part (ii), 	let  $1\leq n_1\neq n_2\leq N$ and  $h\in \oplus^N L^2(\G).$ Then  we have  $P_{n_1}(P_{n_2}h)=0,$  where $P_{n_1}(P_{n_2}h)$  is equal to
{\small	
	\begin{align*}
		\int_{\mc N} \int_{\Gamma}\langle P_{n_2}h, P_{n_2}^*\mc L_\gamma(\oplus_{n=1}^N\varphi_{t}^{(n)})\rangle P_{n_1}(\mc L_\gamma\oplus_{n=1}^N\psi_t^{(n)})  \ d{\mu_\Gamma}(\gamma)\   d{\mu_\mc N}(t)
		=\int_{\mc N} \int_{\Gamma} \langle P_{n_2}h, L_\gamma\varphi_{t}^{(n_2)}\rangle  L_\gamma\psi_t^{(n_1)}\  d{\mu_\Gamma}(\gamma)\   d{\mu_\mc N}(t).
	\end{align*}
}
Hence, $\mc E^\Gamma (\{\varphi_t^{(n_2)}\}_{t\in \mc N})$ and $\mc E^\Gamma (\{\psi_t^{(n_1)}\}_{t\in \mc N})$ are orthogonal pair. Therefore (ii) follows. 
	
	Conversely, let us assume (i) and (ii) hold. Then notice that the both $\{\mc L_\gamma (\oplus_{n=1}^N \varphi_t^{(n)})\}_{t \in \mc N,\gamma \in \Gamma}$ and $\{\mc L_\gamma (\oplus_{n=1}^N \psi_t^{(n)})\}_{t \in \mc N,\gamma \in \Gamma}$  are Bessel families in    $\oplus^N L^2(\G),$ follows by the below calculations for   $h\in L^2(\G)^N$ using  the Bessel property of $\{L_\gamma \varphi_t^{(n)}\}_{t \in \mc N,\gamma \in \Gamma}$  with Bessel  bound $B^{(n)}$:
{\small	\begin{align*}
	\int_{\mc N} \int_{\Gamma} & |\langle h, \mc L_\gamma(\oplus_{n=1}^N \varphi_{t}^{(n)})\rangle |^2\ d{\mu_\Gamma}(\gamma) \  d{\mu_\mc N}(t)  =\int_{\mc N} \int_{\Gamma}  |\langle \oplus_{n=1}^{N}P_n h, \mc L_\gamma(\oplus_{n=1}^N \varphi_{t}^{(n)})\rangle|^2 \ d{\mu_\Gamma}(\gamma)\   d{\mu_\mc N}(t)\\
&	=\int_{\mc N} \int_{\Gamma} \Big|\sum_{n=1}^N\langle P_n h,  L_\gamma \varphi_t^{(n)}\rangle\Big|^2\ d{\mu_\Gamma}(\gamma)  \ d{\mu_\mc N}(t)
	\leq C   \|h\|^2 \sum_{n=1}^N B^{(n)},
	\end{align*}}for some constant $C>0$ (similarly for $\{L_\gamma \psi_t^{(n)}\}_{t \in \mc N,\gamma \in \Gamma}$ ).   Thus  we have the result using Theorem \ref{T:dual-orthogonal-frame-nondiscrete}, by just  looking the  reproducing formula  for  each $h\in \oplus^N L^2(\G)$  and writing $h=\oplus_{n=1}^N P_nh$ in the below calculations:
{\small	\begin{align*}
			&	\int_{\mc N} \int_{\Gamma} \langle h,  \mc L_\gamma(\oplus_{n=1}^{N}\psi_t^{(n)})\rangle \mc L_\gamma(\oplus_{n=1}^N \varphi_{t}^{(n)}) \ d{\mu_\Gamma}(\gamma) \ d{\mu_\mc N}(t) =	\int_{\mc N} \int_{\Gamma}  \sum_{n=1}^{N}\langle P_nh,   L_\gamma\psi_t^{(n)}\rangle \mc L_\gamma(\oplus_{n=1}^N \varphi_{t}^{(n)})\  d{\mu_\Gamma}(\gamma)\  d{\mu_\mc N}(t)\\
		&	=	\int_{\mc N} \int_{\Gamma} \sum_{n=1}^{N} \langle P_nh,   L_\gamma\psi_t^{(n)}\rangle  L_\gamma \varphi_{t}^{(1)}\ d{\mu_\Gamma}(\gamma)  \ d{\mu_\mc N}(t)\oplus\cdots\oplus \int_{\mc N} \int_{\Gamma} \sum_{n=1}^N \langle P_nh,   L_\gamma\psi_t^{(n)}\rangle  L_\gamma \varphi_{t}^{(N)}\ d{\mu_\Gamma}(\gamma)  \ d{\mu_\mc N}(t)\\
		&=P_1h\oplus\dots\oplus P_Nh=h.
	\end{align*}}	

\end{proof}

\section{Applications}\label{App}
In this section, we explore how our findings can be put to use. Since there was  always  an attraction of researches to find various properties of Gabor systems, for instance   \cite{arefijamaal2013continuous, bownik2007SMI, christensen2016introduction, han2000frames, Cabrelli2015SMI, iverson2015subspaces, jakobsen2016co} and references therein, firstly we focus on the Gabor system.

\subsection{Gabor System}  Let $\mc G$ be  a second countable LCA group  having    a closed subgroup $\Lambda.$ Then for  a  family of functions     
	$\mc A=\{\varphi_t:t\in\mc N\}$   in $L^2 (\mc G),$   a \textit{Gabor system}   $G (\mc A, \Lambda, \Lambda^\perp)$ is   
	$$
	G(\mc A, \Lambda, \Lambda^\perp)=\left\{ L_{\lambda} E_\omega \varphi_t : \lambda \in \Lambda, \omega\in \Lambda^\perp, t\in \mc N\right\}, 
	$$ 
	where $\mc N$ is an $\sigma$-finite measure space, and for $\omega \in \widehat  {\mc G},$     the \textit{modulation operator}   $E_{\omega}$ on $L^2(\mc G)$   is defined by $ (E_{\omega}f)(x)=\omega(x)f(x),$  $x\in \mc G, \ f \in L^2 (\mc G).$  We denote $\mc S(\mc A, \Lambda, \Lambda^\perp) :=\overline{\Span}   G(\mc A, \Lambda, \Lambda^\perp).$

In case of  $\mc N$ having counting  measure and  discrete subgroup  $\Lambda,$ the following result is established for the pair $(\mc G, \Lambda)$ by observing the Gabor system $G (\mc A, \Lambda, \Lambda^\perp)$ as   an $\Lambda$-TG system $\mc E^\Lambda (\tilde{\mc A}),$ where  $\tilde{\mc A}=\{E_{\omega} \varphi: \varphi \in \mc A, \omega\in \Lambda^\perp\}.$ The similar results  can be deduced for the  case of uniform lattice $\Lambda$ in $\mc G,$ in particular,  $\mathbb Z^m$ in $\mathbb R^m.$ The following result has a predecessor \cite[Proposition 3.6]{gumber2019pairwisebuletin}.
	
	\begin{thm}\label{T:Gaborcountable} Let $\mc A=\{\varphi_t\}_{t\in \mc N}$ and  $\mc A'=\{\psi_t\}_{t\in \mc N}$ be   sequences  in $L^2(\mc G)$ such that   the Gabor systems   $G(\mc A, \Lambda, \Lambda^\perp)$ and $G(\mc A', \Lambda, \Lambda^\perp)$ are Bessel, where $\Lambda$ is a closed discrete subgroup of an LCA group $\mc G.$   Then     $G(\mc A', \Lambda, \Lambda^\perp)$ is an $\mc S(\mc A, \Lambda, \Lambda^\perp)$-subspace    dual to $G(\mc A, \Lambda, \Lambda^\perp))$ if and only if for all $t' \in \mc N,$  
			$$   
			\mc Z\varphi_{t'} (\beta, x\Lambda)=\sum_{t\in \mc N} [\varphi_{t'},  \psi_t ](\beta) \mc Z\varphi_t (\beta, x\Lambda) \ \mbox{for}\ x\Lambda \in \mc G/\Lambda\ \mbox{and}\ a.e. \ \beta \in \widehat\Lambda.		$$
			In particular, $G(\{\psi\}, \Lambda, \Lambda^\perp)$ is an $\mc S(\{\varphi\}, \Lambda, \Lambda^\perp)$-subspace dual (orthogonal) to $G(\{\varphi\}, \Lambda, \Lambda^\perp)$ if and only if  $[\varphi,  \psi] (\beta)=1$ ($[\varphi,  \psi](\beta)=0$) for a.e. $\beta \in \Omega_{ \varphi}.$
			 	\end{thm}	
	\begin{proof} This follows from the Theorem \ref{T:dualityMultiGenerators} due to the relation  $[ E_\omega\varphi_{t}, E_\omega\psi_t](\beta)=	[ \varphi_{t}, \psi_t](\beta)$ for a.e. $\beta \in \widehat{\Lambda}$ and $t, t' \in \mc N,$ since the Zak transform satisfies  the    formula ${\mc Z}(E_{\omega}f)(\beta, x\Lambda) =\omega(x) \mc Zf(\beta, x\Lambda)$   for 
		$\omega \in \Lambda^\perp,$   $(\beta, x\Lambda) \in (\widehat{\Lambda}, \mc G/\Lambda)$ and $ f \in L^2 (\mc G).$ 	The remaining part follows by Corollary \ref{C:DualMain}.	 
	\end{proof}
 For the arbitrary closed subgroup $\Lambda$ and $\sigma$-finite measure space $\mc N,$ the following result can be deduced    for the set $B$  defined by $B:=\{ (\beta, x\Lambda) \in \widehat \Lambda\times  \mc G/ \Lambda:  \mc Z f (\beta, x\Lambda) \neq 0 \ \mbox{for some } \ f \in \mc A_0 \},$
 where for a given $\mc A$ in $L^2 (\mc G),$ the family of functions $\mc A_0 \ss \mc A$ is a countable dense subset of $\mc A$ ()see \cite{bownik2007SMI,Cabrelli2015SMI,iverson2015subspaces}).  The range function $J_{ \mc A}(\beta, x\Lambda)=\mathbb C.$ 
\cite[Theorem 7.3]{bownik2007SMI} The following result has so many predecessor to the Euclidean setup including \cite[Theorem 1]{bownik2007SMI}, \cite[Theorem 11]{gabardoDual2004Balian}, etc.
\begin{thm}\label{TMI-Alternate Dual Thm}  Let $\mc A=\{\varphi_t\}_{t\in \mc N}$  and $\mc A'=\{\psi_t\}_{t\in \mc N}$ be two collections of functions  in $L^2(\mc G)$   such that   $G(\mc A, \Lambda, \Lambda^\perp)$  and $G(\mc A', \Lambda, \Lambda^\perp)$ are Bessel,   where    $(\mc N,\mu_{\mc N})$  is a complete,   $\sigma$-finite measure space. Then    the following are true:  
	\begin{itemize}
		\item[(i)] $G(\mc A', \Lambda, \Lambda^\perp)$ is an $\mc S(\mc A, \Lambda, \Lambda^\perp)$-subspace dual to $G(\mc A, \Lambda, \Lambda^\perp)$ if and only if for a.e. $(\beta, x\Lambda )\in B,$  the system $({\mc Z} \mc A') (\beta, x\Lambda) :=   \{{\mc Z}\psi_t (\beta, x\Lambda): t \in \mc N\}$ is a $J_{ \mc A}(\beta, x\Lambda)$-subspace dual to  $({\mc Z} \mc A) (\beta, x\Lambda) :=   \{{\mc Z}\varphi_t (\beta, x\Lambda): t \in \mc N\}.$
			\item[(ii)] $G(\mc A', \Lambda, \Lambda^\perp)$ is an $S(\mc A, \Lambda, \Lambda^\perp)$-subspace orthogonal to $G(\mc A, \Lambda, \Lambda^\perp)$ if and only if for a.e. $(\beta, x\Lambda )\in B,$  the system $({\mc Z} \mc A') (\beta, x\Lambda)$ is a  $J_{ \mc A}(\beta, x\Lambda)$-subspace orthogonal to  $({\mc Z} \mc A) (\beta, x\Lambda).$
	 	\end{itemize}
\end{thm}
 \subsection{On the Euclidian space  $\mathbb R^n$ by the action of integer shifts $\mathbb Z^n$} Let us consider $\mc G=\mathbb R^n$ and $\Lambda=\mathbb Z^n$. Then $\widehat{\mc G}=\mathbb R^n$, $\Lambda^\perp=\mathbb Z^n$ and the fundamental domain for  $\mathbb Z^n$ is $\widehat{\mc G}\backslash \Lambda^\perp=\mathbb T^n$. The fiberization map  $ \mathscr T:L^2(\mathbb R^n)\ra L^2(\mathbb T^n;\ell^2(\mathbb Z^n))$  is defined by
$\mathscr Tf(\xi)=\{\widehat {f}(\xi+k)\}_{k\in \mathbb Z^n}, \xi \in \mathbb  T^n$. Therefore we can  obtain characterization results of subspace duals and orthogonal frames  using $\mathscr T,$ which cover results of references within    \cite{christensen2005generalized, christensen2004oblique, heil2009duals, gabardo2007uniqueness,kim2007pair,weber2004orthogonal}.

\subsection{Splines on LCA groups}  Fix $N \in \mathbb N.$ For an LCA group  $\mc G$ with the uniform lattice $\Lambda$ and the associated fundamental domain $\mho,$ the \textit{weighted $B$-spline of order $N$} is defined by  $B_N=\varphi_1\chi_{\mho}*\dots*\varphi_N\chi_{\mho},$
  where  $\varphi_i \in L^2(\mho)$ for $1 \leq i \leq N.$ Then the system $\mc E^\Lambda(B_N)$ is Bessel \cite{christensen2016introduction}. Similarly, the system   $\mc E^\Lambda(B'_N)$ is also Bessel, where $B_N=\psi_1\chi_{\mho}*\dots*\psi_N\chi_{\mho}$ for $\psi_i \in L^2(\mho)$ with $1 \leq i \leq N.$  Therefore for a.e. $\xi \in \Omega$ (fundamental domain associated with $\Lambda^\perp$ in $\widehat{\mc G}$), we have the following similar to Example \ref{ex:discrete-dual2}:
  \begin{align*}
  [B_N,  B'_N]_{\mathscr T}(\xi)=&\sum_{\lambda \in \Lambda ^{\perp}} \widehat{B_N}(\xi+\lambda) \ol{\widehat{B'_N}(\xi+\lambda)}= \sum_{\lambda \in \Lambda ^{\perp}} \Big(\prod{j=1}^N\widehat{(\varphi_j\chi_{\mho})}(\xi+\lambda) \ol{\widehat{(\psi_j\chi_{\mho})}(\xi+\lambda)}\Big).  
  \end{align*}
By the  Corollary \ref{C:DualMain}, the subspace orthogonality and  duality for the system $\mc E^\Lambda(B'_N)$ associated with  $\mc E^\Lambda(B_N)$ can be described by assigning the values on the above expression either 0 or 1, respectively. Due to the importance of splines in numerous applications, both the Euclidean and the LCA group setups conduct in-depth research on them for the Gabor and Wavelet systems \cite{christensen2016introduction, jakobsen2016co}. 
    
\subsection{For other setup}In the scenario of $p$-adic numbers $\mathbb Q_p$ and Heisenberg groups (semidirect product LCA  groups), we may examine our results for the subspace orthogonality and duality of  Bessel families.  Considering that the Zak transform plays a significant role in describing our results, such as, Theorems \ref{T:dualityMultiGenerators},    \ref{T:Orthogonal}, \ref{Th:dual existence},  and  \ref{Th: Dual-Ortho-Non Discerete},   Corollary \ref{C:DualMain}, we describe below the   Zak transform in these setups, see \cite{arefijamaal2013continuous, iverson2015subspaces} and references therein.

\subsubsection{$p$-adic numbers $\mathbb Q_p$} 
	For a prime number $p,$   the locally compact field of $p$-adic numbers $\mathbb Q_p$ is	$\{\sum_{j=m}^{\infty}c_jp^j:m\in \mathbb Z, c_j\in \{0,1,\dots,p-1\}\}$ in which the associated the $p$-adic norm is $|x|_p=p^{-m}$ for  $x=\sum_{j=m}^{\infty}c_jp^j, c_m\neq 0.$  Indeed, it is an LCA group.  The $p$-adic integers $\mathbb Z_p$ defined by  
	$\{\mathbb Z_p:=\{x\in \mathbb Q_p:|x|_p\leq 1\}\}$ is  a compact open subgroup of $\mathbb Q_p.$ In this setup, the Zak transform is given by $\mc Zf(x,y)=\int_{\mathbb Z_p}f(y+\xi)e^{-2 \pi i x \xi}\ d{\mu_{\mathbb Z_p}}(\xi)$ for $ f\in L^1(\mathbb Q_p)\cap L^2(\mathbb Q_p),$ and $x, y \in \Omega$ which can be extended from $L^2(\mathbb Q_p)$ to $\ell^2(\Omega\times \Omega),$ where $\Omega$ is the  fundamental domain.

  \subsubsection{Semidirect product of LCA groups} For LCA groups $\Gamma_1$ and $\Gamma_2,$ let us consider a locally compact group $\mathscr G_\tau=\Gamma_1 \rtimes_\tau \Gamma_2$ by the semidirect product of $\Gamma_1$ and $\Gamma_2$ with the binary operation
 $(\gamma_1,\gamma_2).(\gamma_1',\gamma_2')=(\gamma_1\gamma_1', \gamma_2\tau_{\gamma_1}(\gamma_2)),$  where $\gamma_1\mapsto \tau_{\gamma_1}$ is a group homomorphism from $\Gamma_1$ to the set of all automorphisms on $\Gamma_2,$ such that $(\gamma_1,\gamma_2)\mapsto \tau_{\gamma_1}(\gamma_2)$ from $\Gamma_1\times \Gamma_2\mapsto \Gamma_2$ is continuous.    Then  the  Zak transform $\mc Z$ is defined by   $\mc Zf(\gamma_1,w)=\int_{\Gamma_2}f(\gamma_1,\gamma_2)\overline{w(\gamma_2)}\delta(\gamma_1)\ d{\mu_{\Gamma_2}}(\gamma_2)$  for $f \in L^1 (\Gamma_1\times_{\tau}\Gamma_2),$  $(\gamma_1, \omega) \in \Gamma_1\rtimes_{\widehat{\tau}}\widehat{\Gamma_2},$  where $\delta$ is a positive homeomorphism on $\Gamma_1$ given by 
 $d{\mu_{\Gamma_2}}(\gamma_2) = \delta(\gamma_1) d{\mu_{\Gamma_2}}({\tau_{\gamma_1}}(\gamma_2)).$ It can be extended 
 from $L^2(\Gamma_1\times_{\tau}\Gamma_2)$ to $ L^2(\Gamma_1\rtimes_{\widehat{\tau}}\widehat{\Gamma_2}).$

\end{document}